\documentclass[a4paper,11pt]{amsart}

\usepackage[top=3.3cm, bottom=3cm, left=2.5cm, right=2.5cm]{geometry}
\usepackage{latexsym,amsmath,amsthm,verbatim,ifthen,amssymb,mathtools,booktabs}

\usepackage{hyperref}
\usepackage{caption}
\usepackage{microtype}
\usepackage{graphicx, adjustbox, multirow, float}

\usepackage{enumerate}
\usepackage{palatino}
\usepackage{mathpazo}

\usepackage[all]{xy}
\usepackage{tikz}
\usepackage{tqft}
\usetikzlibrary{positioning,arrows,cd, patterns}

\newcommand{\id}{{\rm Id}}
\newcommand{\tc}{{\rm TC}}
\newcommand{\cd}{{\rm cd}}
\newcommand{\cat}{{\rm cat}}
\newcommand{\secat}{{\rm secat}}

\newcommand{\SO}{\mathrm{SO}}
\newcommand{\U}{\mathrm{U}}
\newcommand{\SU}{\mathrm{SU}}
\newcommand{\Sp}{\mathrm{Sp}}
\newcommand{\Spin}{\mathrm{Spin}}
\newcommand{\PU}{\mathrm{PU}}
\newcommand{\PO}{\mathrm{PO}}
\newcommand{\GL}{\mathrm{GL}}

\newcommand{\NN}{\mathbb{N}}
\newcommand{\ZZ}{\mathbb{Z}}
\newcommand{\FF}{\mathbb{F}}
\newcommand{\CC}{\mathbb{C}}
\newcommand{\RR}{\mathbb{R}}
\newcommand{\PP}{{\mathcal{P}}}

\DeclareMathOperator{\colim}{colim}

\DeclareMathOperator{\Aut}{Aut}
\newcommand{\im}{\mathrm{Im}\,}

\newcommand{\quot}[2]{\left.\raisebox{.2em}{$#1$}\middle/\raisebox{-.2em}{$#2$}\right.}

\newcommand{\ev}{\mathrm{ev}}

\newcommand{\q}[1]{``#1''}

\newtheorem{theorem}{Theorem}[section]

\newtheorem*{thm}{Theorem}

\newtheorem{proposition}[theorem]{Proposition}
\newtheorem{lemma}[theorem]{Lemma}
\newtheorem{corollary}[theorem]{Corollary}

\theoremstyle{definition}
\newtheorem{definition}[theorem]{Definition}

\theoremstyle{remark}
\newtheorem{example}[theorem]{Example}					
\newtheorem{remark}[theorem]{Remark}

\newtheorem*{Note}{Note}

\numberwithin{equation}{section}

\newcommand{\nvar}[2]{%
	\newlength{#1}
	\setlength{#1}{#2}
}

\nvar{\dg}{0.3cm}
\def\dw{0.25}\def\dh{0.5}
\def\link{\draw [double distance=1.5mm, very thick] (0,0)--}
\def\joint{%
	\filldraw [fill=white] (0,0) circle (5pt);
	\fill[black] circle (2pt);
}
\def\grip{%
	\draw[ultra thick](0cm,\dg)--(0cm,-\dg);
	\fill (0cm, 0.5\dg)+(0cm,1.5pt) -- +(0.6\dg,0cm) -- +(0pt,-1.5pt);
	\fill (0cm, -0.5\dg)+(0cm,1.5pt) -- +(0.6\dg,0cm) -- +(0pt,-1.5pt);
}

\def\robotbase{%
	\draw[rounded corners=8pt] (-\dw,-\dh)-- (-\dw, 0) --
	(0,\dh)--(\dw,0)--(\dw,-\dh);
	\draw (-1,-\dh)-- (1,-\dh);
	\fill[pattern=north east lines] (-1,-1) rectangle (1,-\dh);
}

\begin{document}
	
	\title[On properties of effective $\tc$ and LS-category ]{On properties of effective topological complexity and effective Lusternik-Schnirelmann category}
	
	\author[Z. B\l{}aszczyk]{Zbigniew B\l{}aszczyk}
	\address{Faculty of Mathematics and Computer Science,
		Adam Mickiewicz University, Umultowska 87, 60-479 Pozna\'n, Poland.}
	\email{blaszczyk@amu.edu.pl}
	
	\author[A. Espinosa Baro]{Arturo Espinosa Baro}
	\address{Faculty of Mathematics and Computer Science,
		Adam Mickiewicz University, Umultowska 87, 60-479 Pozna\'n, Poland.}
	\email{arturo.espinosabaro@gmail.com, artesp1@amu.edu.pl}
	
	\author[A. Viruel]{Antonio Viruel}
	\address{Departamento de \'Algebra, Geometr\'{\i}a y Topolog\'{\i}a, Universidad de M\'alaga,
		Campus de Teatinos, s/n, 29071-M\'alaga, Spain}
	\email{viruel@uma.es}
	
	
	\begin{abstract}
		The notion of \textit{effective topological complexity}, introduced by B\l{}aszczyk and Kaluba, deals with using group actions in the configuration space in order to reduce the complexity of the motion planning algorithm. In this article we focus on studying several properties of such notion of topological complexity. We introduce a notion of effective LS-category which mimics the behaviour the usual LS-category has in the non-effective setting. We use it to investigate the relationship between these effective invariants and the orbit map with respect of the group action, and we give numerous examples.  Additionally, we investigate non-vanishing criteria based on a cohomological dimension bound of the saturated diagonal.

	\end{abstract}
	
	\maketitle
	
	\setcounter{tocdepth}{1}
	
	\tableofcontents
	
	
		\section*{\centering Introduction}
\subsection*{The motion planning problem and topological complexity} By giving a solution of the \emph{motion planning problem} on a topological space $X$ we understand providing an algorithm which, given any pair of points \mbox{$(x,y) \in X \times X$}, outputs a path in $X$ with initial point $x$ and terminal point $y$. Whenever considered in the context of robotics, $X$ becomes the so called \textit{configuration space}, or space of all posible states of a mechanical system, and such an algorithm can be interpreted as finding the process of moving the robot from one state to another. Formally, let $PX$ denote the space of continuous paths in~$X$ with the compact-open topology. A \textit{motion planning algorithm} (or~a~\textit{motion planner}) in $X$ is a section of the so called \textit{path space fibration} $\pi \colon PX \to X \times X$ given by $\pi(\gamma) = \big(\gamma(0), \gamma(1)\big)$, i.e. a map $s \colon X \times X \to PX$ such that $\pi \circ s = \textrm{id}_{X\times X}$. 

Naturally, one would hope for a motion planner to be continuous, i.e. that small changes in either of the state points translates into a predictable change of the path followed by the robot. However, this situation is rarely possible: as shown by M. Farber in \cite{Farber03}, such continuous motion planners exist only in contractible spaces. Noticing the relationship between that instability and the topological features of the configuration space, Farber introduced the notion of \emph{topological complexity}, denoted by $\tc$, as a measure of the degree of discontinuity of the motion planning problem. $\tc$ turns out to be a particular case of the more general notion of \textit{sectional category} originally introduced as genus of a fibration by Schwarz in \cite{Schwarz66}, and later on generalized to arbitrary maps, see for example \cite{BernsteinGanea}. Since its inception, the concept of topological complexity has proven to be extremely rich, both as a purely theoretical homotopy invariant related to other interesting notions (such as the Lusternik-Schnirelmann category), and as a tool with many possible applications in robotics, and many different variants of the original definition has sprouted, depending on the specific information each of them measure.  

\subsection*{Symmetries on the motion planning} One of those kinds of ''specific information" concerns the impact of the symmetries that often appear in the configuration spaces, and that one might want to take into account whenever studying the instability of the motion planning problem. Formally, those symmetries are seen as actions of groups on the base topological space $X$ and, as such, this naturally leads to the consideration of equivariant versions of topological complexity. There are several non-equivalent approaches to the matter, such as the \emph{equivariant} topological complexity from H. Colman and M. Grant \cite{GrantColman12}, the \emph{strongly equivariant} $\tc$ developed by A. Dranishnikov in \cite{Dranishnikov15} as a variant of the former, the \emph{invariant} $\tc$ of W. Lubawski and W. Marzantowicz introduced in \cite{LubawskiMarzan14}, or the more recent notions of \emph{effectual} topological complexity (by N. Cadavid-Aguilar, J. González, B. Gutiérrez and C. Ipanaque-Zapata, \cite{Gonzalez21}) and \emph{orbital} $\tc$ (defined by E. Balzer and E. Torres-Giese, \cite{BalzerGiese21}). We will not delve into the definitions, so we just refer the interested reader to the excelent surveys on the subject by Ángel and Colman \cite{AnCol18} or Grant \cite{GrantEquiv24}.

However, these versions of equivariant topological complexity, while mathematically relevant on their own, do not take into account the possibility of ``easing" the task of the motion planning through the use of the symmetries of the space. Consider, for example, the case of a robotic arm with two identical pliers, such as the represented in the figure:
	
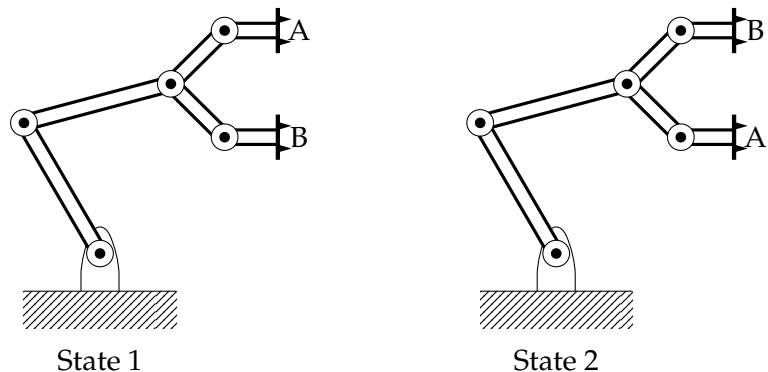
\begin{figure}[H]
	\centering
	\begin{tikzpicture}
		
		\node at (90:-40pt) {State 1};
		\robotbase
		\link(120:2);
		\joint
		\begin{scope}[shift=(120:2), rotate=60]
			\link(-45:2);
			\joint
			\begin{scope}[shift=(-45:2), rotate=-60]
				\link(45:1);
				\link(-45:1);
				\joint
				\begin{scope}[shift=(45:1), rotate=45]
				\link(-45:0.7);
				\joint
				\begin{scope}[shift=(-45:0.7), rotate=-45]
				\grip;
				\node at (1:8pt) {A};
				\end{scope}
				\end{scope}
				\begin{scope}[shift=(-45:1), rotate=45]
					\link(-45:0.7);
					\joint
					\begin{scope}[shift=(-45:0.7), rotate=-45]
						\grip;
						\node at (1:8pt) {B};
					\end{scope}
				\end{scope}
			\end{scope}
		\end{scope}
		
		\begin{scope}[xshift=6cm]
		\node at (90:-40pt) {State 2};
\robotbase
\link(120:2);
\joint
\begin{scope}[shift=(120:2), rotate=60]
	\link(-45:2);
	\joint
	\begin{scope}[shift=(-45:2), rotate=-60]
		\link(45:1);
		\link(-45:1);
		\joint
		\begin{scope}[shift=(45:1), rotate=45]
			\link(-45:0.7);
			\joint
			\begin{scope}[shift=(-45:0.7), rotate=-45]
				\grip;
				\node at (1:8pt) {B};
			\end{scope}
		\end{scope}
		\begin{scope}[shift=(-45:1), rotate=45]
			\link(-45:0.7);
			\joint
			\begin{scope}[shift=(-45:0.7), rotate=-45]
				\grip;
				\node at (1:8pt) {A};
			\end{scope}
		\end{scope}
	\end{scope}
\end{scope}
		\end{scope}
	\end{tikzpicture}
	\caption{A mechanical arm in physically different, but functionally equivalent
		states, since grips A and B are indistinguishable.}
	\label{robotarm}
\end{figure}

	Observe that, while both states are different, any object that has to be manipulated by the arm can be grabbed equally well in both cases. Situations equivalent to this one are extremely common in the world of mechanical systems. The problem is that the original approach to topological complexity does not take this sort of phenomena into account. And yet, the example above suggests that symmetries in configuration spaces can simplify the task of motion planning, given that, even though symmetric positions are physically different, they can be considered as functionally equivalent. Therefore every planning algorithm instructing a robot how to move between all possible states is a waste of effort, and can be made easier if we take into account this functional equivalences. In order to study this possibilites, B\l{}aszczyk and Kaluba introduced in \cite{BlKa2} a new invariant, with precisely this foundational idea, which they baptised as \emph{effective} topological complexity. In light of this notion, the effective motion planners considered in this context output paths that are tipically no longer continuous, but with discontinuities parametrized by the symmetries of the configuration space. As such, whenever a mechanical system follows such a path and runs into a point of discontinuity, it re-interprets its position accordingly within a batch of symmetric positions, and then resumes normal movement.
	
	\subsection*{Contents and structure of this article} So far, the effective topological complexity remains as poorly understood variant of $\tc$. The purpose of this article is to contribute to the understanding of said invariant, by investigating some of its properties. In Section 1 we recall the notion of effective topological complexity, and we review most of its basic properties. In Section 2 we start our analysis by studying the relationship between the different broken path spaces, and introducing the notion of the global effective path space, as a limit of a chain of inclusions. Section 3 is devoted to the definition of an effective version of the Lusternik-Schnirelmann category, which will play an analogous role to the classic LS-category in this setting. In particular, we obtain the effective version of the classic bound of $\tc$ in terms of LS-category, and other crucial properties summarized in the following theorem.
	
	\begin{thm}[Theorem \ref{CatTCEffIneq}, Proposition \ref{Catmap} and Corollary \ref{catTC0}]
		Let $X$ be a $G$-space. The following statements hold:
		\begin{enumerate}[(1)]
			\item $ \cat^{G, \infty}(X) \leq \tc^{G, \infty}(X) \leq \cat^{G \times G, \infty}(X \times X) \leq 2 \cat^{G, \infty}(X).$
			\item Let $\rho_{X} \colon X \rightarrow X/G$ be the orbit map with respect to the action of $G$. Then $\cat(\rho_X) \leq \cat^{G, \infty}(X).$
			\item $\cat^{G,\infty}(X) = 0$ if and only if $\tc^{G, \infty}(X) = 0.$
		\end{enumerate}
	\end{thm}
	
 In Section 4 we will turn our attention to discussing the problem of determining the kind of $G$-spaces with $\tc^{G, \infty}(X) = 0$. The relationship hinted there between the orbit projection map $\rho_X \colon X \rightarrow X/G$ and the effective $\tc$ will be adressed further in Section 5. Here we study the relationship between $\rho_X$ and $\tc^{G,\infty}(X)$ in two distinct cases: when the orbit projection map has a strict section, and when it is a fibration. Our findings are summarized in the following theorem.
	
	\begin{thm}[Theorem \ref{EfOrbSect} and Theorem \ref{EffOrbFibr}]
		
		Let $X$ be a $G$-space. 
		
		\begin{enumerate}[(1)]
		\item If $\rho_X : X \rightarrow X/G$ has a strict section $s : X/G \rightarrow X$, the following holds:
		\begin{enumerate}[(a)]
			\item $\cat^{G, \infty}(X) = \cat(X/G)$.
			\item $\tc^{G, \infty}(X) = \tc(X/G)$.
		\end{enumerate}  
		\item If the orbit map $\rho_X$ is a fibration, then we have:
		\begin{enumerate}[(a)]
			\item $\cat^{G, \infty}(X) = \cat^{G, 2}(X) = \cat(\rho_X) \leq \cat(X/G)$.
			\item $\tc^{G, \infty}(X) = \tc^{G,2}(X) \leq \tc(X/G)$.
		\end{enumerate}
		\end{enumerate}
	\end{thm}
	
	We will discuss plenty of examples in both situations, mostly concerning actions of compact Lie groups, and some consequences of the above result. 
	
In our final Section 6 we will show how the broken path space at stage two $\mathcal{P}_2(X)$ is homotopically equivalent to the saturated diagonal $\daleth(X)$, and we will make use of this information to derive some cohomological conditions for the non-vanishing of the $2$-effective topological complexity for compact $G$-ANR with $G$ finite. 
	
	\begin{thm} [Theorem \ref{PropDimDiag1}, Corollary \ref{TC2great0}]
	Let	$G$ be a finite group, and $X$ a compact $G$-ANR such that $\cd(X^H) \leq \cd(X)$ for all non-trivial subgroup $H \leqslant G$. Then, for any $L$ list of elements of $G$, $\cd(\daleth_L(X)) \leq \cd(X) + |L| - 1$. In particular, we have  $$ \cd(\daleth(X)) \leq \cd(X) + |G| - 1. $$	Under these assumptions, if $|G| \leq \cd(X)$, then it holds that $\tc^{G,2}(X) > 0$. 
	\end{thm}

\bigskip	
\begin{Note}
	In what follows, and unless something different is explicitely stated, we will always consider well pointed spaces with CW-complex structure, and all the actions of the groups considered will be taken as cellular actions. 
\end{Note}
	
	\noindent\textbf{Convention on $\tc$.} In the literature there exists two approaches to topological complexity, namely the non-reduced and the reduced ones, depending whether $\tc$ corresponds with the exact number of local sections or with said number minus one, respectively. In this paper, all the topological complexities are considered as reduced, whereas all the calculations in \cite{BlKa2} are made in the non-reduced setting.
	
	\section*{Acknowledgements}
	The second author was partially supported by Polish National Science Center research grant UMO-2022/45/N/ST1/02814, by Excellence Initiative - Research University project 021-13-UAM-0051 and by a doctoral scholarship of Adam Mickiewicz University. The third author was partially supported by the Spanish Ministerio de Ciencia e Innovaci\'on, project PID2020-118753GB-I00, and Andalusian
	Consejer{\'\i}a de Universidad, Investigaci\'on e Innovaci\'on, project PROYEXCEL-00827.
	
	The authors wish to express their gratitude to Stephan Mescher for his many useful commentaries on an early version of the manuscript, which helped improving the quality of the text, and for bringing to our attention the application of Theorem \ref{EffOrbFibr} to the setting of frame bundles of Riemannian manifolds (see Corollary \ref{TCFrameBundles}).
	
	\section{\centering Review of effective TC} \label{Section1}
	
	Recall that the \textit{sectional category} of a map $f \colon X \to Y$, written $\secat(f)$, is defined to be the smallest integer $n \geq 0$ such that there exists an open cover $U_0$, \ldots, $U_n$ of~$Y$ and continuous maps $s_i \colon U_i \to X$ with the property that $f \circ s_i$ is homotopic to the inclusion $U_i \hookrightarrow Y$ for any $0 \leq i \leq n$. Further on we will make extensive use of both the notion of sectional category and some basic properties of it, so let us recall some of the most useful ones. The following proposition summarizes the basic properties of the sectional category of a fibration. For details on the proofs, we refer the interested reader to the seminal paper on the topic by Schwarz, \cite{Schwarz66}.
	
	\begin{theorem}\label{TheoremSecatProperties}
		Let $F \rightarrow E \xrightarrow{p} B$ a fibration. The following statements hold:
		
		\begin{enumerate}
			\item $\secat(p) \leq \cat(B)$.
			\item Let $k > 0$ the maximal integer such that there exist $$ u_1, \cdots, u_k \in \ker \{ \tilde{H}^*(B;R) \xrightarrow{p^*} \tilde{H}^*(E;R) \} $$ with $u_1 \cup \cdots \cup u_k \neq 0$. Then $\secat(p) \geq k$.
			\item Let $p \colon E \rightarrow B$ be a fibration, and $f \colon X \rightarrow B$ a map. Consider the pullback fibration $f^*p$ over $B$. Then $$ \secat(f^*p) \leq \secat(p). $$
			In particular, if $p \colon E \rightarrow B$ is a fibration and $h \colon B' \rightarrow B$ is a homotopy equivalence, then $\secat(p) = \secat(p')$ for $p'$ the induced fibration of $h$ by $p$. 
			\item Given two fibrations $p \colon E \rightarrow B$ and $p' \colon E \rightarrow B$ consider their product 
			$$ p \times p' \colon E \times E' \rightarrow B \times B.'$$ Then it holds $$ \secat(p \times p') \leq \secat(p) + \secat(p'). $$
		\end{enumerate} 
		
	\end{theorem}
	
	Recall that, given a topological space $X$, we define the path space of $X$, denoted by $PX$, as the space of continuous maps $\gamma \colon I \rightarrow X$, where $I$ here denotes the unit interval in $\RR$. For any element $x \in X$, denote by $c_x$ the constant map in $x$, i.e. $c_x(t) = x$ for all $t \in I$. Denote by $P_* X$ the based path space, which is defined as the restriction of $PX$ to paths starting at an a priori fixed point $x_0 \in X$. Define the fibration $$ \begin{tikzcd}[row sep=0pt,column sep=1pc]
		\ev_1 \colon P_* X \arrow{r} & X  \\
		{\hphantom{(\gamma_0) \colon{}}} 
		\gamma \arrow[mapsto]{r} &  \gamma(1).
	\end{tikzcd} $$  Then, $\secat(\ev_1)$ becomes the well-known \textit{Lusternik-Schnirelmann category} of $X$, denoted by $\cat(X)$. 
	
	Conversely, if the fibration corresponds to the Serre path space fibration, i.e. $$ \begin{tikzcd}[row sep=0pt,column sep=1pc]
		\pi \colon PX \arrow{r} & X \times X \\
		{\hphantom{(\gamma_0) \colon{}}} 
		\gamma \arrow[mapsto]{r} &  (\gamma(0),\gamma(1))
	\end{tikzcd} $$ then $\secat(f)$ coincides with the \textit{topological complexity} of $X$, denoted by $\tc(X)$. \bigskip
	
	We will devote this section of the article to provide a quick review of our main notion of interest, that of effective topological complexity. As such, we recall both its construction and the most useful properties that were first proved in \cite{BlKa2}.
	
	Given a topological group $G$ acting on a pointed CW-complex $X$, and $k \geq 1$ an integer, define the \textit{$k$-broken path space} by  $$ \mathcal{P}_k(X) = \{ (\gamma_1, \cdots, \gamma_k) \in (PX)^k \mid G\gamma_i(1) = G \gamma_{i+1}(0) \mbox{ for } 1 \leq i \leq k \}. $$ In particular, for stages one and two we have the obvious equalities $$\mathcal{P}_1(X) = PX \qquad \mathcal{P}_2(X) = PX \times_{X/G} PX.$$ Denote by $\rho_X : X \rightarrow X/G$ the projection of $X$ onto its orbit space, and by $\delta_X : \daleth(X) \rightarrow X \times X$ the inclusion of the saturated diagonal into $X \times X$. Recall that the saturated diagonal corresponds with the subset $$ \daleth(X) = \{ (g_1 x, g_2 x) \in X \times X \mid g_1, g_2 \in G \mbox{ and } x \in X \}. $$ In section 7 we will see another characterization of the saturated diagonal as a union of ``slices'' indexed by the elements of the group. 
	
	Now we need to define the generalized path space fibrations that encapsulate the desired information about the symmetries in the configuration space. Define the map $ \pi_k : \mathcal{P}_k(X) \rightarrow X \times X $ by $$ \pi_k(\gamma_1, \cdots, \gamma_k) = (\gamma_1(0), \gamma_k(1)).$$ Indeed, this can be seen as a fibration in the following manner: If $\pi_1 :  PX \rightarrow X \times X$ denotes the path space fibration which sends every path to its endpoints, consider the restriction of the fibration $(\pi_1)^k$ to the subspace $X \times \daleth(X)^{k-1} \times X \subseteq (X \times X)^k$. The outcome of this is a fibration $$ p_k : \mathcal{P}_k \rightarrow X \times \daleth(X)^{k-1} \times X $$  fitting to a pullback diagram $$ \begin{tikzcd}
		\mathcal{P}_k(X) \arrow[rrr] \arrow[d, swap, "p_k"] & & & (PX)^k \arrow[d,"(\pi_1)^k"] \\
		X \times \daleth^{k-1} \times X \arrow[rrr, "\text{Id}_X \times (\delta_X)^{k-1} \times \text{Id}_X"] & & & (X \times X)^k.
	\end{tikzcd}$$ Composing $p_k$ with the projection onto the first and last factor, we obtain $\pi_k$. 
	
	\begin{definition}
		With the notation above, define \begin{itemize}
			\item A $(G,k)$-motion planner on an open subset $U \subset X \times X$ is defined as a continuous local section of $\pi_k$ over $U$, that is, a map $s\colon U \rightarrow \mathcal{P}_k(X)$ such that $\pi_k \circ s \simeq \id_U$.
			\item The $k$-stage effective topological complexity $\tc^{G,k}(X)$ as the smallest integer $n \geq 0$ such that there exists an open cover of $X \times X$ by $n+1$ sets admiting $(G,k)$-motion planners. Equivalently, $$ \tc^{G,k}(X) = \secat(\pi_k). $$ 
		\end{itemize} 
	\end{definition}
	
	The following lemma condenses some of the most basic properties of $\tc^{G,k}(X)$ introduced in \cite{BlKa2}: 
	
	\begin{lemma}[{\cite[Lemma 3.2, Theorem 3.3]{BlKa2}}] \label{LemmaMonotEffTC}
		
		$\tc^{G,k}(X)$ satisfies the following properties:
		
		\begin{enumerate}[(1)]
			\item The following inequalities hold for any $k \geq 1$ and any subgroup $H \leq G$: 
			\begin{enumerate}
				\item $\tc^{G,k}(X) \leq \tc^{H,k}(X)$.
				\item $\tc^{G, k+1}(X) \leq \tc^{G,k}(X)$.
			\end{enumerate}
			\item If there exists a $G$-map $f : X \rightarrow Y$ and a map $g : Y \rightarrow X$ such that $f \circ g \simeq \id_Y$ then $$ \tc^{G,k}(Y) \leq \tc^{G,k}(X).$$ In particular, if $X$ and $Y$ are $G$-homotopically equivalent, then $$\tc^{G,k}(X) = \tc^{G,k}(Y).$$
		
		\end{enumerate}
	\end{lemma}
	
	\begin{definition}
		Let $k_0 \geq 1$ be the minimal integer such that $\tc^{G,k}(X) = \tc^{G, k+1}(X)$ for every integer $k \geq k_0$. We define the \textit{effective topological complexity} of $X$ as $$\tc^{G, \infty}(X) = \tc^{G, k_0}(X).$$ 
	\end{definition}
	
	One of the main focus of the original paper from B\l{}aszczyk and Kaluba revolves around providing a full study of the effective topological complexity of $\mathbb{Z}_p$ spheres of any dimension. As bedrock examples that come in handy in many situations, we summarize here the classification provided in \cite{BlKa2}:
	
	\begin{theorem}[{\cite[ Corollary 5.10]{BlKa2}}] \label{EffSpheres}
		For any $p$ prime, suppose $\ZZ_p$ acting on $S^n$ with an $r$-dimensional fixed point set, for $-1 \leq r \leq n-1$ ($r=-1$ meaning free action). 
		
		\begin{itemize}
			\item If $p>2$ then $\tc^{\ZZ_p, \infty}(S^n) = \tc(S^n) = \begin{cases}
				1 & \text{if } $n$ \text{ is odd} \\
				2 & \text{if } $n$ \text{ is even, } n>0.
			\end{cases}$
			
			\item If $p = 2$, then $\tc^{\ZZ_2, \infty}(S^n)$ depends on $r$ as indicated in the following table: \bigskip
			
			\begin{center}
					\begin{tabular} {@{}ccc@{}}\toprule
						& \textsc{O. preserving} & \textsc{O. reversing}\\ \midrule
						$r = -1$ & \multicolumn{2}{c}{$\;\;1$}\\ \cmidrule(rl){1-3}
						\multirow{4}{*}{$0\leq  r \leq  n-2$}
						& \multicolumn{2}{c@{}}{\centering$1$}\\
						& \multicolumn{2}{c@{}}{\footnotesize{for $n$  odd}}\\\cmidrule(lr){2-3}
						& $2$ & $1$\\
						& \footnotesize{for $n$  even} & \footnotesize{for $n$  even, if linear}\\ \cmidrule(rl){1-3}
						\multirow{2}{*}{$r=n-1$}& --- & $0$ \\
						& \footnotesize{not possible}&\footnotesize{if linear}\\\bottomrule
					\end{tabular}
			\end{center} \bigskip
			
			where recall that a group $G$ acts linearly on a $n$-dimensional sphere $S^n$ if there exists a real vector space $V$ of dimension $n+1$, with $S^n$ seen as the unit sphere of $V$, and with a linear action of $G$ such that $S^n$ is a $G$-invariant subspace of $V$.  
			
		\end{itemize}
		
	\end{theorem} \medskip

\begin{remark}
	In\cite{Gonzalez21} N. Cadavid-Aguilar, J. González, B. Gutiérrez and C. Ipanaque-Zapata introduced a tweaked and alternative version of effective topological complexity, different to the original one defined above. They defined a variant of the $k$-broken path space, $Q_k^G(X)$, as the subspace of $(PX \times G)^{k-1} \times PX$ consisting of the tuples $(\alpha_1,g_1,\cdots, \alpha_{k-1},g_{k-1},\alpha_k)$ such that  $\alpha_i(1)\cdot g_i = \alpha_{i+1}(0)$. Defining the $G$-twisted evaluation map 	$$	\begin{tikzcd}[row sep=0pt,column sep=1pc]
		\epsilon_k \colon Q^G_{k} \arrow{r} & X \times X \\
		{\hphantom{(\gamma_0) \colon{}}} (\alpha_1,g_1,\cdots,g_{k-1},\alpha_k) \arrow[mapsto]{r} &  (\alpha_1(0),\alpha_k(1))
	\end{tikzcd} $$ their alternative notion of $k$-effective topological complexity is defined as $$\tc^{G,k}_{\text{effv}}(X) = \secat(\epsilon_k). $$ As it becomes apparent at first glance, the space $Q_k^G(X)$ is designed to encode the precise leaping that assembles a broken path, and this constitutes the main difference with the original notion. Indeed, $$\tc_{\text{effv}}^{G,k}(X) = \tc_{\text{effv}}^{G,k+1}(X) $$ for every $k \geq 2$ and moreover \begin{equation}\label{TwoeffectivesTC}
		\tc^{G,k}(X) \leq \tc_{\text{effv}}^{G,k}(X). 
	\end{equation}    The main advantage of this approach lies in its significant simplification with respect to the original definition. However, while conceptually interesting on its own, this alternative vision of effective $\tc$ does not adress the foundational idea of reducing to the minimum possible the complexity of the motion planning problem through the use of its symmetries, save for nice enough cases, as inequality \ref{TwoeffectivesTC} shows. In fact, this simplification might be a tad excessive for some purposes, as the original notion of effective $\tc$ may fall non-trivially below stage $2$ (we illustrate a basic example later on, see Proposition \ref{ExSeqSphInvol}). Consequently, we think that the further development of the original notion continues to be a worthwile enterprise.  
	
	It is also interesting to remark that, based upon this alternative notion of effective topological complexity, Balzer and Torres-Giese introduced in \cite{BalzerGiese21} the sequential version of effective $\tc$ in the sense of \cite{Gonzalez21}, which coincides, at stage two, with $\tc^G_{\text{effv}}$. 
	
\end{remark}

	\section{\centering The global effective path space} \label{Section2}


We will start our study on the properties of the effective topological complexity by taking a look in this section at the broken path spaces themselves. In particular, we define a notion of ``final" or global effective path space, encompassing the information of all the broken path spaces for each stage $k \geq 0$.

Consider, for each integer $k \geq 0$, the inclusion $\PP_k(X) \xhookrightarrow{\iota_k} \PP_{k+1}(X)$ defined by $$ \iota_0(x) = c_x \in PX $$ and, for every $k > 0$ $$ \iota_k((\gamma_1, \cdots, \gamma_k)) = (\gamma_1, \cdots, \gamma_k, c_{\gamma_k(1)}) \in \PP_{k+1}(X). $$ As one could easily expect from the definition, these inclusions are well behaved. 

\begin{lemma}
	For each integer $k \geq 0$, the inclusion $\PP_k(X) \xhookrightarrow{\iota_k} \PP_{k+1}(X)$ embeds the broken path space $\PP_k(X)$ as a closed subspace of $\PP_{k+1}(X)$. 
\end{lemma}
\begin{proof}
The case $k = 0$ is straightforward. For any $k \geq 1$ define the projection map that sends any broken path in $\PP_{k+1}(X)$ to its last component as a path in $PX$, i.e. $$\varphi_{k+1} \colon \PP_{k+1}(X) \rightarrow PX \qquad \varphi_{k+1}(\gamma_1, \cdots, \gamma_{k+1}) = \gamma_{k+1}. $$ It is immediate to check that this is a well defined continuous map, and such that $$ \phi_{k+1} =  p_{k+1_{\vert_{\PP_{k+1}(X)}}} $$ where $p_{k+1} \colon (PX)^{k+1} \rightarrow PX$ is just the obvious projection map into the $(k+1)$-coordinate.
Now, consider the obvious inclusion $$ i \colon X \rightarrow PX \qquad x \mapsto c_x.$$ Given that $X$ is taken to be a Hausdorff space, $PX$ is Hausdorff as well and, as we assumed $X$ to be compact, $i(X)$ is compact in $PX$ and, therefore, closed. Observe that $$ \iota_{k}(\PP_{k-1}(X)) = \varphi_{k+1}^{-1}(i(X))$$ and thus the claim follows from the continuity of $\varphi_{k+1}$.
\end{proof}

\begin{definition}
 We define the \textit{global effective path space}, denoted by $\PP_{\infty}(X)$ as  $\colim \PP_k(X)$ with respect to the chain of inclusions $$ X \xhookrightarrow{\iota_0} PX \xhookrightarrow{\iota_1} \PP_2(X) \xhookrightarrow{\iota_2} \cdots \PP_k(X) \xhookrightarrow{\iota_k} \PP_{k+1}(X) \xhookrightarrow{\iota_{k+1}} \cdots $$ endowed with the final (colimit) topology. 
\end{definition}

It is clear that we can visualize the building at each stage of the global effective path space through the chain of inclusions considered above as a sort of \q{cellular attachment}, in the following sense: for each integer $n \geq 0$, the broken path space at stage $n$ can be seen as fitting the pushout diagram 
$$ \begin{tikzcd}
	\overline{\PP_{n-1}(X)} \ar[r, "\varphi_{n-1}"] \arrow[hook]{d} & \PP_{n-1}(X) \arrow[d, " \iota_{n-1}"] \\
	\PP_n(X) \setminus \PP^{\circ}_{n-1}(X) \ar[r] & \PP_n(X) 
\end{tikzcd} $$ where $\PP_{-1}(X) = \emptyset$ and obviously $ \PP_0(X) \setminus \PP^{\circ}_{-1}(X) = \PP_0(X) = X $.

Given that the global effective path space $\PP_{\infty}(X)$ is endowed with the weak (colimit) topology, any subset $U \subset \PP_{\infty}(X)$ is open (respectively closed) if and only if, for every $n \geq 0$, the intersection $U \cap \PP_n(X)$ is open (respectively closed). 

	\begin{proposition}
		Let $ S \colon F \rightarrow \PP_{\infty}(X) $ be a continuous map. If $F$ is both Hausdorff and compact, then $S$ factors through $\PP_n(X)$ for some integer $n$.
	\end{proposition}
	\begin{proof}
		Take the image $S(F) \subset \PP_{\infty}(X)$ as a compact subset. Let us assume that the statement is false. As such, there exists an infinite subset of integers $\mathcal{J}$ such that the intersection $S(F) \bigcap ( \PP_k(X) \setminus \PP_{k-1}(X) )$ is non-empty for every $k \in \mathcal{J}$. Now, for each $k \in \mathcal{J}$, take exactly one distinguished element of such intersection $$ x_k \in  S(F) \bigcap (\PP_k(X) \setminus \PP_{k-1}(X) ) $$ which defines a sequence of integers $J = \{ x_k \}_{k \in \mathcal{J}}$ of infinite length, such that that $x_n \neq x_m$ if $n \neq m$ by its definition. 
		
		Given that $J \subset \PP_{\infty}(X)$ with induced topology, any subset of $J$ is open (alternatively closed) in $\PP_{\infty}(X)$ if and only if its intersection $J \cap \PP_k(X)$ is open/closed in $\PP_k(X)$ for each $k \geq 0$. Now, for each $ x_k \in J$ notice that $$ \{x_k\} \cap \PP_r(X) = \begin{cases}
			\emptyset & \text{if }  r < k \\
			\{x_k\} & \text{if }   r \geq k.
		\end{cases} $$ It is clear that $\PP_r(X)$ and $\PP_{\infty}(X)$ are Hausdorff spaces, given that $PX$ is Hausdorff, and $\PP_r(X)$ and $\PP_{\infty}(X)$ can be seen as a subspaces of (in)finite products of copies of $PX$. As such, the one point set $\{x_k \}$ is closed. For any $r \in \mathcal{J}$ the intersection $J \cap \PP_r(X) = \{ x_{k_i} \}_{k_i \in \mathcal{J}}$ for $k_i \leq r$, and therefore $J \cap \PP_r(X)$ is expressible as a finite union of closed subsets for every $r$, hence is closed. This implies that $J$ is a closed subset of $\PP_{\infty}(X)$. Now, it is  straightforward to show that every subset of $J$ is closed, and consequently $J$ is equipped with the discrete topology which, by the hypothesis of compacity, contradicts the assumption. \end{proof}
	\medskip
		Consider now each broken path space $\PP_k(X)$ as a $G^k$-space via the component-wise action. The space~$X \times X$ has a natural structure as a $(G\times G)$-space, but it can be seen also as a $G^k$-space via precomposition of the $(G\times G)$-action with the projection $G^k \to G\times G$ onto the first and last coordinates. In this manner $\pi_k \colon \PP_k(X) \to X \times X$ becomes a $G^k$-equivariant map, and one obtains the following commutative diagram. 
		$$ \begin{tikzcd}
			\PP_k(X) \arrow[rr, "\pi_k"] \arrow[d] & & X \times X \arrow[d] \\
			\PP_k(X)/G^k \arrow[rr] \arrow[rd] & & X/G \times X/G \\
			& P(X/G) \arrow[ru]
		\end{tikzcd}$$
	
	Here the vertical maps are orbit projections, the lower horizontal map is induced by $\pi_k$, the oblique map on the left is the concatenation of a sequence of~$k$ paths in $X/G$, and the oblique map on the right is the path space fibration for $X/G$. Composing the orbit projection for $\PP_k(X)$ and the concatenation of paths in $\PP_k(X)/G^k$, we get the obvious commutative diagram
	
	$$	\begin{tikzcd}
		\PP_k(X) \rar \arrow[d, swap, "\theta_k"] & X \times X \arrow[d, "\rho_{X \times X}"] \\
		P(X/G) \rar  & X/G \times X/G.
	\end{tikzcd} $$
	\begin{proposition}\label{EffOrbitSect}
		Let $X$ be a $G$ space and suppose that, for some $k > 0$, there exists a continuous map $$ \overline{s_k} \colon P(X/G) \rightarrow \PP_k(X)$$ such that $ \theta_k \circ \overline{s_k} = \id_{P(X/G)}$. Then $\tc^{G, k+2}(X) \leq \tc(X/G)$ and therefore $$\tc^{G, \infty}(X) \leq \tc(X/G).$$   
	\end{proposition}
	\begin{proof}
		Let $n := \tc(X/G)$. Consider $\{ V_i \}_{0 \leq i \leq n}$ with $V_i \subset X/G \times X/G$ and $s_i \colon V_i \rightarrow P(X/G)$ a local section for the path space fibration $\pi \colon P(X/G) \rightarrow X/G \times X/G$ for every $0 \leq i \leq n$.

		Define for each $0 \leq i \leq n$ the open set $U_i = (\rho_X \times \rho_X)^{-1}(V_i)$. The map $\overline{s_k}$ restricted to $U_i$ is of the form $$ \overline{s_k}(\overline{\gamma}) = (\gamma_1, \cdots, \gamma_k) $$ with the obvious condition $G \gamma_j(1) = G\gamma_{j+1}(0)$ for every $0 \leq j \leq k$. Now define, for each $U_i$, a map $$ \xi_i \colon U_i \rightarrow \PP_{k+2}(X) \qquad \xi_i(x,y) = (c_x, \overline{s}_k(\overline{\gamma}), c_y) $$ for $\overline{\gamma}(0) = [x]$ and $\overline{\gamma}(1) = [y]$. It is immediate from its definition that $\pi_k \circ \xi_i = \id_{U_i}$ and thus $\{ U_i\}_{0 \leq i \leq n}$ constitutes a categorical cover for $\tc^{G,k+2}(X)$. Consequently $$\tc^{G,k+2}(X) \leq n = \tc(X/G).$$ \end{proof}
	
	\begin{remark}
		It is important to note that the section $\overline{s_k}$ assumed before is not necessarily induced by a section $s \colon X/G \rightarrow X$ of the orbit map $\rho_X$. We will see that, if such a section $s$ exists, then Proposition \ref{EffOrbitSect} is just an immediate consequence of Theorem \ref{EfOrbSect}, that we will state and prove in a later section. 
		
	\end{remark}
	
	\section{\centering Effective LS-category} \label{Section3}
	
	In most cases topological complexity is a significantly difficult invariant to compute, one for which no general systematic way of calculation exists. One of the possible approaches to give estimates for $\tc$ relies in its well known bounds by Lusternik-Schnirelmann category which is, in most cases, an easier invariant to compute, and it is generally better understood than its counterpart.
	
	There is a natural version of LS-category in the equivariant setting, the so called Lusternik-Schnirelmann $G$-category, which is defined as follows. For a $G$-space $X$ we say that a $G$-invariant open subset $U \subseteq X$ is \emph{$G$-categorical} if the inclusion $U \hookrightarrow X$ is $G$-homotopic to a $G$-equivariant map with values in a single orbit. Then, the LS $G$-category of $X$, denoted by $\cat_G(X)$, is defined as the smallest integer $m \geq 0$ such that there exists an open cover of $X$ by $m+1$ $G$-categorical open subsets. Indeed, for both equivariant and invariant topological complexity, a lower bound in terms of LS $G$-category can be found, at least in cases where the set of fixed points is non-empty, see \cite[Proposition 5.7, Corollary 5.8]{GrantColman12} and \cite[Proposition 2.7]{BlKa1} respectively. However, as discussed by Z. B\l{}aszczyk and M. Kaluba in \cite[Section 7]{BlKa2}, such lower bound is not possible for effective topological complexity. Moreover, they noted that this impossibility do not stem from the particular definition of the invariant, but rather from the philosophy behind it, i.e. such bound would be impossible to accomplish for any other homotopy invariant $\mathcal{TC}$ with the property $\mathcal{TC}(X) \leq \tc^{G,\infty}(X)$. In fact, such an anomalous behaviour in the effective setting is hardly surprising. After all, unlike the cases of (strongly) equivariant and invariant $\tc$, the effective motion planners are not required to be equivariant. 
	
	Given the additional layer of difficulty that the effective topological complexity carries, it is only natural to ponder whether a category lower bound can be laid down in the effective setting. The unfeasibility of the LS $G$-category points out to the necessity of considering a new candidate, an analogue of usual LS-category for the effective setting. In this section we will fill such void, and we will develop a notion of effective Lusternik-Schnirelmann category, which we will show that behaves analogously in the effective setting as the classic LS-category does in the classic one.
	\medskip
	
Recall that, given a fibration $f: X \rightarrow Y$, property (1) in Theorem \ref{TheoremSecatProperties} gives the upper bound $\cat(Y) \geq \secat(f)$. Let $x_0 \in X$ such that $P_*X$ is the space of paths starting at $x_0$, and consider the inclusion $X \hookrightarrow X \times X$ by $x \mapsto (x_0,x)$. There is an obvious pullback diagram of the form $$\begin{tikzcd}
	P_*X \arrow[r] \arrow[d, swap, "\ev_1"] & PX \arrow [d, "\pi_1"] \\
	X \arrow[r, hook] & X \times X
\end{tikzcd}$$ and taking into account that $\cat(X) = \secat(\ev_1)$, we have the classic chain of inequalities relating category and $\tc$ 
\begin{equation} \label{estandarcatTc}
	\cat(X) \leq \tc(X) \leq \cat(X \times X) \leq 2 \cat(X).
\end{equation} 

We can further generalize the previous pullback diagram considering an analogous pullback diagram associated, for each $k > 0$, to the $k$-effective fibration $\pi_k$ :

$$\begin{tikzcd}
	P^k_*(X) \arrow[r] \arrow[d, swap, "q_k"] & \PP_k(X) \arrow [d, "\pi_k"] \\
	X \arrow[r, hook] & X \times X.
\end{tikzcd}$$
In this way we obtain a fibration $q_k : P^k_*(X) \rightarrow  X$ as a pullback of $\pi_k$ by the inclusion of $ X \hookrightarrow X \times X$. Those are, precisely, the fibrations that encode in the effective setting the relationship analogous to the one the usual LS-category had with the standard $\tc$. Thus, the definition comes naturally. 
\begin{definition}
	For $k \geq 1$ integer we define the $k$-effective Lusternik-Schnirelmann category of $X$ as $\cat^{G,k}(X) = \secat(q_k)$. The \textit{effective LS-category} of $X$, thus, is defined as $$ \cat^{G, \infty}(X) = \text{min} \{ \cat^{G, k}(X) \mid k \geq 1 \}. $$
\end{definition}


It is straightforward from the definition that $q_1 = \ev_1$, so $\cat^{G,1}(X) = \cat(X)$. Indeed, the classic chain of inequalities relating LS-category and topological complexity \ref{estandarcatTc} can be generalised to the effective setting in a natural way:
	
	\begin{theorem}\label{CatTCEffIneq}
		For $X$ a $G$-space, the following chain of inequalities holds: 
		\begin{equation} \label{effectivecatTc}
			\cat^{G, \infty}(X) \leq \tc^{G, \infty}(X) \leq \cat^{G \times G, \infty}(X \times X) \leq 2 \cat^{G, \infty}(X).
		\end{equation} 
	\end{theorem}
	\begin{proof}
		For the first inequality, observe that $q_k$ is defined as a pullback fibration of the $k$-effective fibration $\pi_k$, so the inequality holds by $(3)$ in Theorem \ref{TheoremSecatProperties}. 
		
		To show the second inequality, first notice that we can immediately identify $P^k_*(X \times X) = P^k_*(X) \times P^k_*(X)$. Consider a categorical cover $\{ U_j\}_{0 \leq j \leq n}$ of $X \times X$ for $q_k$, take any of its open subsets $U_i \subset X \times X$ and a local section $s_{U_j}$ of $q_k$ over $U$, defined as $$ s_{U_j} := ((s_1, \cdots, s_k),(s'_1, \cdots s'_k)) $$ where, for each $1 \leq i \leq k$, the entries $s_i$ and $s'_i$ correspond with components of the local section to the $i$-th coordinate of $P_*^k(X)$ for each of the two copies of $X$ in the cartesian product. Define now a map from $U_i$ to the $(2k-1)$-broken path space $$ \sigma_{U_j} : U_j \rightarrow \PP_{2k-1}(X) \qquad \sigma_{U_j}(x,y) = (s_k(x,y)^{-1}, \cdots, s_1(x,y)^{-1} \ast s'_1(x,y), \cdots, s'_k(x,y)) $$ where, for each index $1 \leq i \leq k$, we denote by $s_i(x,y)^{-1}$ the path walked in reverse orientation, and $s_1(x,y) \ast s'_1(x,y)$ is just the corresponding concatenation of paths. One checks that this map determines a local section for the fibration $\pi_{2k-1}$ over $U_j$ for each of the possible choices of $U_j$ in the categorical cover, hence $$\tc^{G,\infty}(X) \leq \cat^{G, \infty}(X \times X).$$
		
		Finally, the last inequality is just a consequence of property $(4)$ in Theorem \ref{TheoremSecatProperties}.
	\end{proof} 
	
	From the definition and Theorem \ref{TheoremSecatProperties} it is obvious that, in analogy with the effective topological case \begin{equation} \label{effcatleqcat}
		\cat^{G, \infty}(X) \leq \cat^{G,k}(X) \leq \cat(X).
	\end{equation} As such, combining \ref{effcatleqcat}, Proposition \ref{effectivecatTc}, Corollary \ref{catTC0} and Theorem \ref{EffSpheres}, we immediately derive the effective LS category of $\ZZ_p$-spheres.
	
	\begin{corollary}
		For any prime $p$, suppose $\ZZ_p$ acting on $S^n$. Then $$ \cat^{\ZZ_p, \infty}(S^n) = 1. $$
	\end{corollary}
	
	Recall that by the LS-category of a map $f: X \rightarrow Y$ we understand the minimal number of open sets in a covering of $X$ such that $f$ is nullhomotopic over each one of them. It is not surprising that the category of the orbit map of $X$ with respects to the $G$-action turns out to be a lower bound for the effective LS category of $X$:
	
	\begin{proposition}\label{Catmap}
		Let $X$ be a $G$-space, and $\rho_{X} \colon X \rightarrow X/G$ the orbit map with respect to the action of $G$. Then $\cat(\rho_X) \leq \cat^{G, \infty}(X).$
	\end{proposition}
	\begin{proof}
		Define a map $\lambda \colon P^k_*(X) \rightarrow P_*(X/G)$ by composing each path component in $P_*^k(X)$ with the orbit map $\rho_X \colon X \rightarrow X/G$, and then concatenating the resulting paths in $X/G$ in the order prescribed by their appearance in the $k$-tuple. Explicitly put $$\begin{tikzcd}[row sep=0pt,column sep=1pc]
			\lambda \colon P_*^k(X) \arrow{r} & P_*(X/G) \\
			{\hphantom{(\gamma_0) \colon{}}} (\alpha_1, \cdots, \alpha_k) \arrow[mapsto]{r} &  (\rho_X \circ \alpha_1) \ast \cdots \ast (\rho_X \circ \alpha_k).
		\end{tikzcd} $$ The map $\lambda$ thus defined fits inside a commutative diagram of the form $$ \begin{tikzcd}
			P_*^k(X) \arrow[r, "\lambda"] \arrow[d, swap, "q_k"] &  P_*(X/G) \arrow[d, "q_1"] \\
			X \arrow[r, swap, "\rho_X"]  & X/G. 
		\end{tikzcd} $$ Consider now $\{ U_i \}_{0 \leq i \leq n}$ an effective categorical open cover of $X$ for $\cat^{G, \infty}(X)$, and take for each index $0 \leq i \leq n$ a local section $s_i: U_i \rightarrow P^k_{\ast}(X) $ of $q_k$. Define for each $0 \leq i \leq n$ a map $$ H_i : U_i \times I \rightarrow X/G \qquad  H(x,0) = (\lambda \circ s_i)(x)(0)  \qquad H(x,1) = (\lambda \circ s_i)(x)(1) = \rho_{| U_i}(x). $$ Notice that $H_i$ defines a a homotopy between $\rho_{X_{| U_i}}$, the restriction of the orbit map on $U_i$, and a constant map, making $\rho_{X_{| U_i}}$ nullhomotopic and therefore showing that  $\{U_i\}_{0 \leq i \leq n}$ is a categorical cover for $\cat(\rho_X)$, so it follows $$ \cat(\rho_X) \leq \cat^{G,\infty}(X). $$ \end{proof}
	
	Let us discuss an example on how to make use of the notion of effective LS-category to bound effective topological complexity from below. 
	
	\begin{example}
		Consider $\CC P^n \times \CC P^n$ with $\ZZ_2$ acting on the product by switch of coordinates. It is clear that the action is not free, but it can be turned into a free action in a standard way by applying the Borel construction. That way, by considering the orbit projection with respect to this induced free action we end up with a $2$-fold covering projection map $$ \rho \colon \mathbb{C}P^n \times \mathbb{C}P^n \times S^{\infty} \rightarrow (\mathbb{C}P^n \times \mathbb{C}P^n \times S^{\infty})/\ZZ_2. $$  The category of a map is bounded below by the nilpotency of its image in cohomology, so we have $\cat(\rho) \geq \text{nil}(\im \rho^*)$. Recall that the real cohomology ring structure of $\mathbb{C}P^n$ corresponds with $$ H^*(\mathbb{C}P^n; \mathbb{R}) = \mathbb{R}[\alpha]/(\alpha^{n+1}) \qquad |\alpha| = 2.$$ If we denote by $x$ and $y$ the generators of the second cohomology group of $\mathbb{C}P^n \times \mathbb{C}P^n$ corresponding to the factors of the product then by \cite[Proposition 3G.1]{Hatcher} we have that $x + y \in \text{Im}\rho^*$. Given that $$ (x+y)^{2n} = \binom{2n}{n} x^n y^n \neq 0 $$ we obtain that $\text{nil}(\text{Im}\rho^*) \geq 2n$, and so it follows that $$ \tc^{\ZZ_2, \infty}(\mathbb{C}P^n\times \mathbb{C}P^n) \geq \cat^{\ZZ_2, \infty}(\mathbb{C}P^n\times \mathbb{C}P^n) \geq \cat(\rho) \geq 2n. $$

	\end{example}

	\section{\centering The problem of $\tc^{G,\infty}(X) = 0$} \label{Section4}

	It is well known since the inception of the whole theory, \cite[Theorem 1]{Farber03} that the only  spaces with topological complexity equal to zero are those which are contractible. Perhaps, it is not so surprising that, given the additional layer of complexity that is involved in the definition of the effective variant, such a basic case it is still unknown. In this section we will briefly discuss the situation, and also present counter-examples to certain proposed characterizations of $\tc^{G,\infty}(X) = 0$.  
	
	It is immediate to check from the definition that, by design, if $X$ is a contractible or $G$-contractible space, then $\tc^{G,\infty}(X)=0$. The converse, however, is not true, and an easy counterexample can be constructed by considering the computation of the $\ZZ_p$-spheres of Theorem \ref{EffSpheres}: 
	
	\begin{example}
		Consider the unit $n$-sphere $S^n$, for $n \geq 1$, equipped with a $\ZZ_2$ action by involution, which interchanges the two hemispheres and leaves the equator fixed. By Theorem \ref{EffSpheres} we have that $\tc^{\ZZ_2, \infty}(S^n) = 0$.
	\end{example} 
	
	Despite the failure of this reciprocity, the condition that the effective topological complexity of a $G$-space is zero imposes a strong condition over orbit map with respect to the action, as the following proposition shows.
	
	\begin{proposition} \label{OrbitNullHom}
		Let $X$ be a $G$-space such that $\tc^{G,\infty}(X)=0$. Then the orbit projection map $\rho_X \colon X \to X/G$ is nullhomotopic.
	\end{proposition} 
	\begin{proof}
		Assume that $\tc^{G, \infty}(X) = 0$. Then there is an integer $k \geq 0$ such that there exists a global section of the $k$-effective fibration $\pi_k$, i.e. a  map $$ s \colon X \times X \rightarrow \PP_k(X) \qquad  \pi_k \circ s \simeq \id_{X \times X}.$$ Defining the map $$\zeta_k \colon X \times X \rightarrow P(X/G), \qquad \zeta_k :=\theta_k \circ s,$$ we obtain the following commutative diagram $$ \begin{tikzcd}
			& P(X/G) \arrow[d]  \\
			X \times X \arrow[r, swap, "\rho_X \times \rho_X"] \arrow[ru, "\zeta_k"]  & X/G \times X/G.
		\end{tikzcd} $$ Therefore, for a choice of a distinguished point $x_0 \in X$, we can define a map $$\begin{tikzcd}[row sep=0pt,column sep=1pc]
			H \colon X \times I \arrow{r} & X/G \\
			{\hphantom{(\gamma_0) \colon{}}} (x,t) \arrow[mapsto]{r} &  \zeta_k(x_0, x)(t)
		\end{tikzcd} $$ which, evaluated at $t = 0$ and $t = 1$, gives the following values
		
		$$ H(x,0) = \zeta_k(x_0,x)(0) = \rho_X(x_0), \qquad H(x,1) = \zeta_k(x_0, x)(1) = \rho_X(x). $$ Hence, $H$ defines a homotopy between the orbit map $\rho_X$ and the constant path $c_{\rho_X(x_0)}$, and as a result $\rho_X$ is seen to be nullhomotopic. \end{proof}	 
	
	Unfortunately, the converse of the previous implication, again, does not hold in general. This time the counterexample is a little bit more elaborated, though: 
	
	\begin{example}
	Consider the six-dimensional sphere, $S^6$, with $\ZZ_2$ acting on it via the antipodal action. Now, take the orbit projection map $\rho \colon S^6 \rightarrow \mathbb{R}P^6$. The eighth suspension of the orbit projection $$ \Sigma^8 \colon \Sigma^8 S^6 \to \Sigma^8 \mathbb{R}P^6$$ coming from the antipodal action on $S^6$ can be seen to be nullhomotopic (by the work of E. Rees in his PhD thesis, see \cite[Corollary 2]{Rees69}). However, $\Sigma^8 S^6$, equipped with the corresponding involution has a $7$-dimensional fixed point set. By Theorem \ref{EffSpheres}, we have that $\tc^{\ZZ_2, \infty}(\Sigma^8 S^6) = 1$.
	\end{example}
	
	In the previous section, we made use of our definition of effective LS-category to generalize the classic bound of topological complexity in terms of Lusternik-Schnirelmann category, see Theorem \ref{CatTCEffIneq}. It is important to notice that one of the immediate consequences of such upper and lower bound indicates an alternative approach to the problem of determining the kind of $G$-spaces with effective topological complexity equal to zero.
	
		\begin{corollary} \label{catTC0}
		If $X$ is a $G$-space, $\cat^{G,\infty}(X) = 0$ if and only if $\tc^{G, \infty}(X) = 0.$
	\end{corollary}
	
	\section{\centering Effective topological complexity and the orbit projection} \label{Section5}
	
	It is only natural to ponder about the relationship between the effective topological complexity of a $G$-space and distinguished properties of the orbit projection map associated to the $G$-action. In this section, we will investigate the influence of two of such properties. First, we analyze the scenario where the orbit projection map is endowed with a strict section. After that, we consider the instance where the orbit map is a fibration. In both cases, plenty of examples of computations and bounds are given.
	
	\subsection{Orbit map has a strict section.}
	
	In the circumstance that the orbit projection by the group action is equipped with a strict section $s \colon X/G \rightarrow X$, the effective framework gets significantly simplified. By using this section, one can lift all paths in $X/G$ to paths in the base space $X$, and the effective LS category and topological complexity coincide with the corresponding non-effective ones of the orbit space.
	
	\begin{theorem}\label{EfOrbSect}
		Let $X$ be a $G$-space. If the orbit map $\rho_X : X \rightarrow X/G$ has a strict section $s : X/G \rightarrow X$, the following holds:
		
		\begin{enumerate}[(1)]
			\item $\cat^{G, \infty}(X) = \cat(X/G)$.
			\item $\tc^{G, \infty}(X) = \tc(X/G)$.
		\end{enumerate}  
	\end{theorem}
	\begin{proof}
		Let us prove the second claim. Start by considering an open cover  $\{ U_i \}_{0 \leq i \leq n}$ with $U_i \subset X/G \times X/G$ and a local section $\sigma_i : U_i \rightarrow P(X/G) $ of $\pi_1$ for each $0 \leq i \leq n$. Now put $V_i := (\rho_X \times \rho_X)^{-1}(U_i)$ an open set in $X \times X$, and consider the map induced at the level of path spaces by the section $s$, i.e. $$\begin{tikzcd}[row sep=0pt,column sep=1pc]
			\overline{s} \colon P(X/G) \arrow{r} & PX \\
			{\hphantom{(\gamma_0) \colon{}}} \overline{\gamma} \arrow[mapsto]{r} & \overline{s}(\overline{\gamma})(t) = s(\overline{\gamma}(t)). 
		\end{tikzcd} $$ Now we can define a local section of the effective fibration $\pi_3\colon \PP_3(X) \rightarrow X \times X$, denoted $\varsigma_i : V_i  \rightarrow \PP_3(X)$, by the expression $$ \varsigma_i(x,y) := (c_x, \overline{s}[\sigma_i([x],[y])], c_y). $$ This shows that $\tc^{G, \infty}(X) \leq \tc(X/G)$. 
		
		For the reverse inequality, let $n := \tc^{G,\infty}(X)$, and consider an open cover $\{ V_i \}_{0 \leq i \leq n}$ of $X \times X$, and $\varsigma_i : V_i \rightarrow \PP_k(X)$ as a continuous local section for the effective fibration $\pi_k \colon \PP_k(X) \rightarrow X \times X$ for some $k >0$ realizing $\tc^{G,\infty}(X)$. Define $\overline{\rho_x} \colon \PP_k(X) \rightarrow P(X/G)$, a map induced in $\PP_k(X)$ by the orbit map, by projecting any $k$-broken path $\gamma = (\gamma_1, \cdots, \gamma_k) \in \PP_k(X)$ through the orbit map, and concatenating, for each $1 < j < k$, the end point of $\rho_X(\gamma_j)$ with the initial point of $\rho_X(\gamma_{j+1})$, i.e. $$ \overline{\rho_X}(\gamma_1, \cdots, \gamma_k) = (\rho_X \circ \gamma_1) \ast \cdots \ast (\rho_X \circ \gamma_k). $$ Finally, observe that for each $0 \leq i \leq n$ the composite map $$\xi_i := \overline{\rho_x} \circ \varsigma_i \circ (s \times s)$$ defines a local section of $\pi_1 \colon P(X/G) \rightarrow X/G \times X/G$ over $U_i := (s \times s)^{-1}(V_i)$, and so $$ \tc(X/G) \leq \tc^{G,k}(X). $$ 
		
		\medskip
		
		With this approach in mind, the proof of $1.$ is, essentially, analogous. Start by considering $\{ U_i \}_{0 \leq i \leq m}$, a categorical open cover for $\cat(X/G)$.  If we regard $\cat$ as a sectional category we have, for each $0 \leq i \leq m$, a local section $\sigma_i \colon U_i \rightarrow P(X/G)$. Define now, as above, $V_i = (\rho_X \times \rho_X)^{-1}(U_i)$ and a local section for $q_2$ by $$ \varsigma_i(x) := (\overline{s}[\sigma_i([x])],c_x). $$ This shows that $\cat^{G,2}(X) \leq \cat(X/G)$. For the reverse inequality, if we have an open cover $\{ V_i \}_{0 \leq i \leq m}$ and local sections $\varsigma_i \colon V_i \rightarrow P_*^k(X)$ of the fibration $q_k$ for some $k$ realizing $\cat^{G,\infty}(X)$ then, putting $U_i = (s \times s)^{-1}$ we can define a local section of $\ev_1 \colon P(X/G) \rightarrow X/G$ over $U_i$ by $$ \xi_i := \lambda_k \circ \varsigma_i \circ s $$ where $\lambda_k$ is as defined in the proof of Proposition \ref{Catmap}.
	\end{proof}
	
	Let explore some examples of the theorem above:
	
	\begin{example} \label{ExStrictSect}
		
		\begin{enumerate}[(1)]
			\item As an immediate consequence we obtain that $\tc^{\ZZ_2, \infty}(S^n) = 0$ when the action is the flip (i.e. reflection interchanging the hemispheres and fixing the equator). Though this was computed in \cite[Proposition 5.7]{BlKa2}, Theorem \ref{EfOrbSect} provides a more general and conceptual explanation to it. 
			
			\item  Recall that the unitary group $\U(n)$ fits inside a split short exact sequence of groups of the form $$ \SU(n) \hookrightarrow \U(n) \rightarrow \U(1) \cong S^1. $$ Hence, by Theorem \ref{EfOrbSect} $$ \tc^{\SU(n), \infty}(\U(n)) = \tc(S^1) = 1.$$
			
			\item Recall that the special orthogonal group, denoted $SO(n)$, is the group of orthogonal matrices in the $n$-dimensional euclidean space with determinant equal to $1$. The principal bundle $$ \SO(3) \rightarrow \SO(4) \rightarrow S^3 $$ has a section, and consequently $$ \tc^{\SO(3), \infty}(\SO(4)) = \tc(S^3) = 1. $$ Later on we will see more applications of our results to more general cases of $\SO(n)$. 
			
			\item As illustrated in the previous two cases, split Lie group extensions are a rich source of examples for $G$-spaces with strict sections for their orbit map. Other examples of split exact sequences of groups in the spirit of the previous example can be obtained in the following manner: let $p > 2$ be a prime integer, $r \geq 1$ and define the central product $$ S(p^r,p^r)  = \SU(p^r) \times_{\Gamma_{p^r}} \SU(p^r) $$ where $\Gamma_{p^r}$ corresponds with the diagonal cyclic subgroup of the center of order $p^r$. Now one can make $\SU(p^r)$ act on $S(p^r,p^r)$ by left action on just the first coordinate of the central product. Under this action we obtain a principal bundle $$ \SU(p^r) \rightarrow S(p^r,p^r) \rightarrow \PU(p^r) $$ and such bundle has, indeed, a global section. Hence we get $$\tc^{\SU(p^r), \infty}(S(p^r,p^r)) = \cat(\PU(p^r)) = 3(p^r-1)$$ where the last equality was computed in \cite{IMN05}. 
			
			\item Let $X$ be a based space, and $G$ any group. Construct the space $Z = \displaystyle{\vee_{g \in G} X_g}$ defined by $X_g = X$, and equipped with a $G$-action given by $h x_g = x_{hg}$ for $x_g = x \in X_g$. Then $$\tc^{G, \infty}(Z) = \tc(X).$$  
		\end{enumerate}	
		
	\end{example}
	
	The last case of the previous example allows us to give an easy realization result for effective topological complexity:
	
	\begin{corollary}
		Let $G$ any finite group and $n \geq 0$ a non-negative integer. Then there exists a $G$-space $X$ such that $\tc^{G, \infty}(X) = n$.
	\end{corollary}
	
	\begin{proof}
		Consider a space $Y$ with $\tc(Y) = n$ (an easy example is $Y = T^n$). Now, construct the space $X = \displaystyle{\vee_{g \in G} Y_g}$ defined in the same manner as in Example \ref{ExStrictSect} (3) above. As a consequence of Theorem \ref{EfOrbSect} we have that $$ \tc^{G,\infty}(X) = \tc(Y) = n. $$ \end{proof}

	\subsection{Orbit map is a fibration.}
	
	In this case, the situation has richer derivations, but requires a bit more subtlety. The equality obtained in the presence of a strict section is not always possible. However, we can collapse both effective LS category and topological complexity at stage 2, and bound them both by their corresponding non-effective counterparts of the orbit space by the group action, as the following theorem shows. 
	
	\begin{theorem}\label{EffOrbFibr}
		Let $X$ be a $G$-space such that the orbit map $\rho_X \colon X \rightarrow X/G$ is a fibration. Then:
		\begin{enumerate}[(1)]
			\item $\cat^{G, \infty}(X) = \cat^{G, 2}(X) = \cat(\rho_X) \leq \cat(X/G)$.
			\item $\tc^{G, \infty}(X) = \tc^{G,2}(X) \leq \tc(X/G)$.
		\end{enumerate}
	\end{theorem}
	\begin{proof}
		\begin{enumerate}[(1)]
			
			\item Let $\{ U_i \}_{0 \leq i \leq n}$ be a categorical open cover of $X$ for $\cat(\rho_X)$. By the hypothesis of nullhomotopy of $\rho_X$ over every $U_i$, it is possible to construct a family of homotopies of the form $$ H_i : U_i \times I \rightarrow X/G, \qquad  H_i(x,0) = \rho_X(x_0), \qquad H_i(x,1) = \rho_X(x).$$ Since $\rho_X$ is a fibration, by the homotopy lifting property $H_i$ can be lifted through $\rho_X$ to a homotopy $K_i : U_i \times I \rightarrow X$ satisfying $$K_i(x,0) = x_0, \qquad \rho_X \circ K_i = H_i. $$ Define now, for every $0 \leq i \leq n$, a map $$s_i: U_i \rightarrow P^2_{\ast}(X) \qquad \mbox{by } s_i(x) = (K_i(x,\cdot), c_x). $$ It is clear that $s_i$ constitutes a local section for the fibration $q_2 \colon P^2_*(X) \rightarrow X$ over $U_i$, and therefore $$\cat^{G,2}(X) \leq \cat(\rho_X).$$ By Proposition \ref{Catmap}, this means that $$ \cat^{G, \infty}(X) = \cat^{G,2}(X) = \cat(\rho_X)$$ and the last inequality of the claim follows from usual properties of the category of a map.  
			
			\item Now let $\{ U_i \}_{0 \leq i \leq n}$ be an open cover of $X \times X$ such that there exists, for every $0 \leq i \leq n$, a local section $s_i \colon U_i \rightarrow \PP_k(X) $ of the $k$-effective fibration $\pi_k$ over $U_i$, for some $k$ such that $\tc^{G,k}(X) = \tc^{G,\infty}(X)$. Recall from the proof of Theorem \ref{EfOrbSect} that one can define a map $\overline{\rho_X}\colon \PP_k(X) \rightarrow X/G \times X/G$ induced by the orbit projection map $\rho_X$ by $$ \overline{\rho_X}(\gamma_1, \cdots, \gamma_k) = (\rho_X \circ \gamma_1) \ast \cdots \ast (\rho_X \circ \gamma_k). $$ Through the map $\overline{\rho_X}$, every local section $\sigma_i$ defines a homotopy $$ H_i: U_i \times I \rightarrow X/G$$ by putting $$ H_i((x,y),0) = \overline{\rho_X}(s_i(x,y))(0) = \rho_X(x),  \qquad H_i((x,y),1) = \overline{\rho_X}(s_i(x,y))(1) = \rho_X(y). $$ Since $\rho_X$ is a fibration by hypothesis, we have a lifting for $H_i$, the homotopy $$ K_i: U_i \times I \rightarrow X $$ satisfying $$ K_i((x,y),0) = x, \qquad \rho_X \circ K_i = H_i. $$ Through this homotopy it is possible to define a local section $\sigma_i \colon U_i \rightarrow \PP_2(X)$ of the effective fibration $\pi_2$ over $U_i$, by putting $$ \sigma_i(x,y) := (K_i((x,y),\cdot),c_y) $$ which shows that $\tc^{G,\infty}(X) = \tc^{G,2}(X)$. 
			
			For the last inequality, consider $\{ V_i \}_{0 \leq i \leq m}$ an open cover of $X/G \times X/G$ such that there exists, for each $0 \leq i \leq m$, a local section over $V_i$ of the path space fibration $\pi \colon P(X/G) \rightarrow X/G \times X/G$. Define for each $i$ a homotopy $$ P_i: V_i \times I \rightarrow X/G$$ satisfying $P_i(([x],[y]),0) = [x]$ and $P_i(([x],[y]),1) = [y]$ for each $([x],[y]) \in V_i$, and put $W_i := (\rho_X \times \rho_X)^{-1}(V_i)$. There are induced homotopies $$ W_i \times I \xrightarrow{(\rho_X \times \rho_X)\times \id_{I}} V_i \times I \xrightarrow{P_i} X/G $$ and, since $\rho_X$ is a fibration, we can lift them to obtain new homotopies $$Q_i \colon W_i \times I \rightarrow X$$ such that $$ \rho_X \circ Q_i = P_i \circ ((\rho_X \times \rho_X)_{\vert_{ W_i}} \times \id_I) $$ and consequently $$ Q_i((x,y),0) = x \qquad Q_i((x,y),1) = z \in [y].$$ Through this last family of homotopies, a local section $\lambda_i: W_i \rightarrow \PP_2(X) $ for the effective fibration $\pi_2$ can then be defined by putting $$ \lambda_i(x,y) := (Q_i((x,y),\cdot),c_y)$$ thus $ \tc^{G,2}(X) \leq \tc(X/G)$.
		\end{enumerate} \end{proof}
	Whenever $G$ is a discrete group acting properly discontinuously on $X$, Theorem \ref{EffOrbFibr} recovers the bound of effective topological complexity by $\tc(X/G)$ of \cite[Theorem 1.1]{Gonzalez21}. However, the situation is much more interesting if we are considering actions of compact Lie groups.
	\begin{example}
		\begin{enumerate}[(a)]
			\item 	Under the identification of $S^1$ as the topological unitary group $U(1)$, we have a very well known fibre bundle $$ S^1 \hookrightarrow S^{2n+1} \rightarrow \mathbb{C}P^n. $$ As a consequence of (2) of Theorem \ref{EffOrbFibr}, we see that $$ \tc^{S^1, \infty}(S^{2n+1}) \leq \tc(\CC P^n) = 2n $$ (where the value of $\tc(\CC P^n)$ was computed in \cite{Farber2003}). In this case, however, the bound provided by theorem is far from a sharp one. Notice that we can consider the subgroup inclusion $\ZZ_p \leqslant S^1$ and hence, by virtue of Lemma \ref{LemmaMonotEffTC} and Theorem \ref{EffSpheres} one gets $$ \tc^{S^1, \infty}(S^{2n+1}) \leq \tc^{\ZZ_p, \infty}(S^{2n+1}) = 1 $$  and, as a consequence of Proposition \ref{OrbitNullHom}, $\tc^{S^1, \infty}(S^{2n+1}) = 1$.
			
			Furthermore, we can take the principal bundle associated to the classifying space of $U(1)$, $$ S^1 \hookrightarrow E\U(1) \rightarrow B\U(1) \qquad \equiv \qquad S^1 \hookrightarrow S^{\infty} \rightarrow \CC P^{\infty}$$ and in this case the contractibility of $S^{\infty}$ implies that $$ \tc^{S^1, \infty}(S^{\infty}) = 0. $$
			
			\item It is well known that the identification map  (sometimes called ''realification") $$ \phi \colon \CC^{n\times n} \rightarrow \RR^{2n \times 2n}$$ given by putting $$ C := A + iB \mapsto \begin{bmatrix}
				A & -B \\[1ex]
				B & A 
			\end{bmatrix} $$ allows to identify the linear group $\U(n)$ as a subgroup of $\SO(2n)$. To be more specific, it can be shown that $$ \phi(\U(n)) = \SO(2n) \cap \phi(\GL(n, \CC)). $$ There is then a principal $\U(3)$-bundle $$ \U(3) \hookrightarrow \SO(6) \rightarrow \CC P^3 $$ which, in conjunction with Theorem \ref{EffOrbFibr} informs us that $$ \tc^{\U(3), \infty}(\SO(6)) \leq \tc(\CC P^3) = 6. $$
			
			\item Think of the $S^{2n+1}$ sphere immersed in the $(n+1)$-dimensional complex space $\CC^{n+1}$. Recall that the map $T \colon \CC^{n+1} \rightarrow \CC^{n+1}$ defined as the scalar multiplication by the $p$-th root of unity, (i.e. $T(z) = \exp(2 \pi i/p)z$ for $z \in S^{2n+1}$) generates the standard complex representation of the cyclic group $\ZZ_p$. This induces a free action on $S^{2n+1}$ under a complex unitary map of period $p$. The orbit space of such action is the well known lens space $L^{2n+1}_p$. 
			
			We can consider, in $S^{2n+1}$, the scalar multiplication of $z \in S^{2n+1}$ by $\exp(2\pi ix/p)$, where $x \in \RR$. This operation defines a group homomorphism $g \colon \RR \rightarrow \Aut(S^{2n+1})$ which commutes with the induced periodic map $T \colon S^{2n+1} \rightarrow S^{2n+1}$. Consequently, $g$ induces an action of $\RR$ on the lens space $L^{2n+1}_p$ through an induced homomorphism $\overline{g} \colon \RR \rightarrow \Aut(L^{2n+1}_p)$. If we take an integer $k$, it is easy to see that $\exp(2\pi ik/p) = (\exp(2 \pi i/p)^k)$, which informs us that the integers act trivially on $L^{2n+1}_p$. Therefore, the map $\overline{g}$ factors through the exponential map and it subsequently induce an action of $S^1$, regarded as the circular group, on the lens space $L^{2n+1}_p$, defined explicitely as $$ s\cdot [z] = [\exp(2\pi ix/p)z]$$ for $z \in S^{2n+1}$, $[z] \in L^{2n+1}_p$ and $s \in S^1$, $x \in \RR$ such that $s = \exp(2\pi i x)$. Jaworowski, in \cite{Jaworowski99}, demonstrated that such an action is free and, furthermore, that the orbit space under it corresponds with the complex projective space $\CC P^n$. Therefore, by Theorem \ref{EffOrbFibr} we see that $$ \tc^{S^1, \infty}(L^{2n+1}_p) \leq \tc(\CC P^n) = 2n. $$

			\item Although the situation is significantly more complicated in the case of real projective spaces, we can still make use of the known topological complexity of $\RR P^n$ for certain values of $n$ to derive even more examples from Theorem \ref{EffOrbFibr}. As it is discussed in \cite[Section 4]{IMN05}, we have the following principal bundles of compact Lie groups over real projective spaces: $$ \Sp(1) \rightarrow \SO(5) \rightarrow \RR P^7, \qquad  \SU(3) \rightarrow \SO(6) \rightarrow \RR P^7, $$ $$ G_2 \rightarrow \SO(7) \rightarrow \RR P^{15}, \qquad \Spin(7) \rightarrow \SO(9) \rightarrow \RR P^{15}, $$ $$ G_2 \rightarrow \PO(8) \rightarrow \RR P^7 \times \RR P^7.  $$ Therefore, by Theorem \ref{EffOrbFibr} and the computation of the topological complexity of real projective spaces in dimension $7$ and $15$ carried out in \cite{Farber2003}, we obtain the inequalities:
			$$ \tc^{\Sp(1), \infty}(\SO(5)) \leq 7, \quad \tc^{\SU(3), \infty}(\SO(6)) \leq 7, $$ $$ \tc^{G_2, \infty}(\SO(7)) \leq 23, \qquad \tc^{\Spin(7), \infty}(\SO(9)) \leq 2, $$ $$ \tc^{G_2, \infty}(\PO(8)) \leq 14. $$	
		\end{enumerate}
	\end{example}
	\medskip
	
	Let $G$ be a matrix Lie group, and $H \leq G$ a closed subgroup. It is a well known fact that $G$ has the structure of a fibre bundle $$ H \hookrightarrow G \xrightarrow{\rho} G/H $$ (see, for example \cite[Proposition 13.8]{Hall15}). In particular, Theorem \ref{EffOrbFibr} produces very easy upper bounds for actions of closed matrix subgroups in their immediate matrix overgroup:
	
	\begin{corollary} \label{MatrixLieGr}
		Let $n \in \NN$. Then the following holds:
		
		\begin{enumerate}[(a)]
			\item $\tc^{\SO(n-1), \infty}(\SO(n)) = 1$ for $n$ even and $\tc^{\SO(n-1), \infty}(\SO(n)) \leq 2$ for $n$ odd. 
			
			\item For $n \geq 2$ we have $\tc^{\U(n-1), \infty}(\U(n)) = 1$.
			
			\item For $n \geq 3$, we have $\tc^{\SU(n-1), \infty}(\SU(n)) = 1$. 
			
			\item For all $n \geq 1$ we have $\tc^{\Sp(n-1), \infty}(\Sp(n)) = 1$.

		\end{enumerate}
		
	\end{corollary}
	
	\begin{proof}
		
		The statements have almost analogous proofs. All of them depend on the identification of the orbit maps with fibrations with base spheres of appropiate dimension (to see a proof of these facts see, for example, \cite[Section 13.2]{Hall15}) and on the computation of the standard topological complexity of spheres (see \cite{Farber08}).

		\begin{enumerate}[(a)]
			\item $\SO(n-1)$ acting over $\SO(n)$ fits into a fibration $$ \SO(n-1) \hookrightarrow \SO(n) \xrightarrow{\rho} \quot{\SO(n)}{\SO(n-1)} \cong S^{n-1}.  $$ By (2) of Theorem \ref{EffOrbFibr} we know that $$ \tc^{\SO(n-1), \infty}(\SO(n)) \leq \tc(S^{n-1}) = \begin{cases}
				1 & \mbox{ for n even} \\
				2 & \mbox{ for n odd}
			\end{cases} $$
			
			\item The action of $\U(n-1)$ on its overgroup $\U(n)$ fits into a principal bundle $$ \U(n-1) \hookrightarrow \U(n) \rightarrow \quot{\U(n)}{\U(n-1)} \cong S^{2n-1} $$ which informs us, by virtue of (2) of Theorem \ref{EffOrbFibr}, that $$ \tc^{\U(n-1), \infty}(\U(n)) \leq \tc(S^{2n-1}) = 1. $$
			
			\item The subgroup $\SU(n-1)$ acting over $\SU(n)$ fits into a fibration of the form 
			$$ \SU(n-1) \hookrightarrow \SU(n) \xrightarrow{\rho} \quot{\SU(n)}{\SU(n-1)} \cong S^{2n-1}. $$ Once again, (2) of Theorem \ref{EffOrbFibr} implies that $$ \tc^{\SU(n-1), \infty}(\SU(n)) \leq \tc(S^{2n-1}) = 1. $$
			
			\item Finally, $\Sp(n-1)$ acting over $\Sp(n)$ makes the orbit map projection into a fibration 
			$$ \Sp(n-1) \hookrightarrow \Sp(n) \xrightarrow{\rho} \quot{\Sp(n)}{\Sp(n-1)} \cong S^{4n-1}. $$ As previously, by (2) of Theorem \ref{EffOrbFibr} $$ \tc^{\Sp(n-1), \infty}(\Sp(n)) \leq \tc(S^{4n-1}) = 1. $$
			
		\end{enumerate}
		
		Finally note that, as a consequence of Proposition \ref{OrbitNullHom}, the previously determined upper bounds by $1$ are, indeed, sharp equalities. \end{proof}

	For any pair of numbers $n,k \in \NN$, with $k < n$, we say that a (compact) \textit{Stiefel} manifold over a field $\FF \in \{ \mathbb{R}, \mathbb{C}, \mathbb{H} \}$, denoted by $V_k(\FF^n)$, is the set of $k$-orthonormal tuples of vectors in $\FF^n$, with the subspace topology in $\FF^{n+k}$. Conversely, a $k$-Grassmannian over $\FF^n$ is the set of all possible $k$-dimensional vector subspaces of $\FF^n$. Now recall that the $k$-orthogonal group over $\FF$, denoted by $O(k,\FF)$, is defined as the subgroup of the $k$-general linear group on $\FF$ of orthogonal matrices, i.e. $$ O(k,\FF) = \{ Q \in GL(k,\FF) \mid Q^TQ = QQ^T = I\}. $$ The group $O(k,\FF)$ acts freely on $V_k(\FF^n)$,  by rotating a $k$-frame in the space it spans. The orbits of this action are precisely the orthonormal $k$-frames spanning a given $k$-dimensional subspace, that is, the orbit map corresponds with a fibration (indeed a principal $O(k,\FF)$-bundle) of the form $$O(k,\FF) \hookrightarrow V_k(\FF^n) \xrightarrow{\rho} G_k(\FF^n).$$ If we specialize the concrete choice of the field, we obtain fibrations 
	\begin{equation} \label{Grass2}
		O(k) \hookrightarrow V_k(\RR^n) \xrightarrow{\rho} G_k(\RR^n) \qquad U(k,\CC) \hookrightarrow V_k(\CC^n) \xrightarrow{\rho} G_k(\CC^n)
	\end{equation}
	which allow us to give upper bounds for the effective topological complexity of orthogonal actions on Stiefel manifolds.
	\begin{corollary}
		Let $k,n \in \NN$ with $k < n$. Then the following bounds are satisfied:
		\begin{enumerate}[(a)]
			\item $\tc^{O(k), \infty}(V_k(\RR^n)) \leq 2k(n-k)-1  $
			\item $\tc^{U(k), \infty}(V_k(\CC^n)) \leq 2k(n-k)$
		\end{enumerate}
	\end{corollary}
	
	\begin{proof}
		As a consequence of the fibrations in \ref{Grass2} and (2) of Theorem \ref{EffOrbFibr}, we know that $$ \tc^{O(k), \infty}(V_k(\RR^n)) \leq \tc(G_k(\RR^n)) \qquad \mbox{ and } \qquad \tc^{U(k), \infty}(V_k(\CC^n)) \leq \tc(G_k(\CC^n)).  $$ Then both claims follow from the upper bounds for the topological complexity of Grassmann manifolds computed by P. Pave\v{s}i\'c in \cite[Proposition 4.1, Theorem 4.2]{Pavesic21}.
	\end{proof}
	
	 We will briefly recall the definition of more general orthonormal frame bundles defined over smooth manifolds, and we will apply our results to that setting. Let $M$ be an $n$-dimensional (oriented) Riemann manifold, define, for every $x \in M$ the space $$ F_x(M) := \{ (v_1, \cdots, v_n) \in (T_xM)^n \mid (v_1, \cdots, v_n) \mbox{ is a positive orthonormal basis of } T_xM  \}$$ and from it the \emph{space of positive orthonormal frames} of $M$ by putting $$ F(M) := \{ (x,b) \mid x \in M, b \in F_x(M) \}. $$ Build the continuous map $ p_{F(M)} \colon F(M) \rightarrow M$  by $p_{F(M)}(x,b) = x.$ Then $p_{F(M)}$ has the structure of a smooth principal $\SO(n)$-bundle, called the bundle of positive orthonormal frames of $M$. 
	
	
	\begin{corollary}\label{TCFrameBundles}
	Let $M$ a path-connected $n$-dimensional smooth manifold, and $p_{F(M)} \colon F(M) \rightarrow M$ defined as above. Then \begin{enumerate}[(a)]
		\item $\tc^{\SO(n),\infty}(F(M)) \leq 2 \dim(M)$.
		\item Furthermore, if $M$ is parallelizable, then $$ \tc^{\SO(n), \infty}(F(M)) =  \tc(M). $$
	\end{enumerate} 
	\end{corollary}
	\begin{proof}
		\begin{enumerate}[(a)]
			\item Given that $p_{F(M)}$ is a principal bundle, we are under the assumptions of Theorem \ref{EffOrbFibr}, hence $$ \tc^{\SO(n),\infty}(F(M)) \leq \tc(F(M)/\SO(n)) = \tc(M) \leq 2 \dim(M) $$ where the last inequality just comes from the well-known dimensional bound of topological complexity, see \cite[Theorem 5]{Farber03}.
			\item Under the hypothesis of $M$ being parallelizable, $p_{F(M)} \colon F(M) \rightarrow M$ becomes a trivial $\SO(n)$-bundle, thus the claim follows from Theorem \ref{EfOrbSect}.
		\end{enumerate} \end{proof} To check computations of usual topological complexity of orthonormal frame bundles we refer the readers to the analysis of S. Mescher on the matter, see \cite{Mescher19}. \medskip
	
	Under nice enough group actions, the quotient space of a smooth manifold is itself a manifold with a smooth structure making the orbit map a fibration. 
	
	\begin{corollary} \label{smoothact}
		Let $G$ be a Lie group acting smoothly, freely and properly on a connected smooth manifold $M$. Then $$ \tc^{G,\infty}(M) \leq 2(\dim(M) - \dim(G)). $$
	\end{corollary} 
	\begin{proof}
	By the quotient manifold theorem (see \cite[Theorem 21.10]{Lee}) the orbit space $M/G$ has the structure of a topological manifold with $$\dim(M/G) = \dim(M) - \dim(G)$$ and with an unique smooth structure satisfying that the orbit map $\rho_{M} \colon M \rightarrow M/G$ is a smooth submersion. By the Ehresmann's fibration theorem (see \cite[Theorem 8.5.10]{Dundas}) $\rho_{M}$ is a (locally trivial) fibration, hence Theorem \ref{EffOrbFibr} gives us $$ \tc^{G,\infty}(M) = \tc^{G,2}(M) \leq \tc(M/G) \leq 2(\dim(M) - \dim(G)) $$ where the last inequality just follows from the dimensional upper bound of $\tc$. 
	\end{proof}
	
	We can easily obtain the same inequality for locally smooth free actions though, unlike in the case above, we have to impose compacity as an additional restriction. Let $G$ be this time a compact Lie group, acting over a closed connected smooth manifold $M$. Locally smooth actions come equipped with principal orbits, i.e. orbits of type $G/H$ such that $H$ is subconjugated to any isotropy subgroup $G_x \leqslant G$. By virtue of \cite[Theorem IV.3.8]{Bredon72} we have that $$ \dim(M/G) = \dim(M) - \dim(P). $$ If, furthermore, the action of $G$ is taken as free, the orbit projection map $\rho_M \colon M \rightarrow M/G$ becomes a fibration, hence Theorem \ref{EffOrbFibr} applies, and we have $$ \tc^{G,\infty}(M) \leq \tc(M/G) \leq 2 \dim(M/G) $$ and, since $\dim(P) = \dim(G)$, we immediately get \begin{equation}\label{Locsmooth}
		\tc^{G,\infty}(M) \leq 2(\dim(M) - \dim(G)).
	\end{equation} Compare both Corollary \ref{smoothact} and inequality \refeq{Locsmooth} above with the upper bounds for usual topological complexity of smooth manifolds with locally smooth free actions obtained by M. Grant in \cite{GrantFib}.
	
	\medskip
	
	
	It is possible to find conditions under which the effective LS category and the regular LS category coincide in this setting:
	
	\begin{corollary}\label{CorBasePoint}
		For a group $G$, if $X$ is a $G$-space such that the orbit of the base point $G x_0$ is contractible in $X$ then $\cat^{G, \infty}(X) = \cat(X)$.
	\end{corollary}
	\begin{proof}
		Let $\{ U_i \}_{0 \leq i \leq n}$ be a categorical cover for $\cat(\rho_X)$ and define, for each $U_i \subset X$ a homotopy $H_i : U_i \times I \rightarrow X/G$ such that $$ H_i(x,0) = \rho_X(x_0) \qquad \mbox{and} \qquad H_i(x,1) = \rho_X(x). $$ Given that $\rho_X$ is a fibration, we can lift $H_i$ to a homotopy $K_i : U_i \times I \rightarrow X$ satisfying $$ K_i(x,1) = x \qquad \mbox{and} \qquad \rho_X \circ K_i = H_i. $$ By the hypothesis of contractibility of the orbit of the basepoint, there is a continuous map $\theta: G x_0 \rightarrow P_{\ast} X $ such that $\theta(x)(1) = x$ for all $x \in Gx_0$. Consequently, it is possible to define a section for the LS-cat fibration $q_1$ over $U_i$ as the concatenation $$ s(x) = \theta(K_i(x,0)) \ast H_i(x,\cdot).$$ The claim thus follows from the equality $\cat(\rho_X) = \cat^{G,\infty}(X)$ in $(1)$ of Theorem \ref{EffOrbFibr}.
	\end{proof}
	
	Naturally, the previous corollary is of particular interest in situations where we can assume the coincidence between LS-category and topological complexity, so the computation of effective topological complexity would be derived from the (generally easier) task of knowing the classic LS-category.   
	
	\begin{corollary} \label{FreeCatEqTC}
		Let $G$ be a finite group and $X$ a free $G$-space such that $\cat(X) = \tc(X)$. Then $\tc^{G, \infty}(X) = \cat(X)$.
	\end{corollary}
	\begin{proof}
		The chain of inequalities in Theorem \ref{CatTCEffIneq} shows that $\cat^{G,\infty}(X) \leq \tc^{G, \infty}(X) \leq \tc(X)$. By Corollary \ref{EffOrbFibr} above we know that $\cat^{G, \infty}(X) = \cat(X)$, and the claim follows from the hypothesis. 
	\end{proof}

	We will close this subsection by mentioning some interesting examples of the previous corollaries.
	
	\begin{example}
		\begin{enumerate}[(1)]
			
			\item Let $G$ be a connected Lie group. In \cite[Lemma 8.2]{FarberIns04}, Farber proved the equality $\tc(G) = \cat(G)$. Therefore, if there exists a finite non-trivial discrete subgroup $H \leq G$, Corollary \ref{FreeCatEqTC} implies that $$ \tc^{H,\infty}(G) = \cat(G).$$ Recall that, for example, if $G$ is a non-nilpotent and simply connected group (such as groups of upper triangular matrices with diagonal terms equal to one) there always exists a non-trivial finite subgroup $H$. 
			
			\item Generalizing Farber's result, Lupton and Scherer demonstrated in \cite[Theorem 1]{Lupton13} that, if $X$ is a connected CW $H$-space, then $\tc(X) = \cat(X)$. Consequently, if $G$ is a finite group acting freely on $X$, Corollary \ref{FreeCatEqTC} applies and $$ \tc^{G,\infty}(X) = \cat(X). $$
			
			\item A particularly simple example comes from free products on spheres. Let $G$ be a finite group acting freely on the $k$-dimensional torus 
			$$ T^k = \underbrace{S^1 \times \cdots \times S^1}_{k}.$$ Then we have $$\tc^{G, \infty}(T^k) = k.$$ 
			More generally, we can consider products of odd dimensional spheres, so if $G$ acts freely on the product $\underbrace{S^{2n+1} \times \cdots \times S^{2n+1}}_{k}$ we obtain $$\tc^{G,\infty}(\underbrace{S^{2n+1} \times \cdots \times S^{2n+1}}_{k}) = k$$
		
		\end{enumerate}
	\end{example}

		\section{\centering Cohomological conditions for non-vanishing of $\tc^{G,2}(X)$} \label{Section6}
		
		One of the significant open problems suggested in the original article of B\l{}aszczyk and Kaluba \cite{BlKa2} concerned the determination of the kind of sequences that could arise as sequences of effective topological complexities. The problem is too broad and general, and it will most certainly require a specific in-depth inquiry on the matter, which goes beyond the scope of the present article. However, we will make a first contribution to the problem.
		
		In this section we will study some cohomological conditions to determine whether or not the effective topological complexity at stage two vanishes. The set stage is not arbitrary by any means: such cohomological conditions are examined over the saturated diagonal, and we will make use of an homotopy equivalence between $\daleth(X)$ and the stage 2 broken path space $\PP_2(X)$ to infer the aforementioned non-vanishing condition. \medskip
		
		Let us start by noticing that the saturated diagonal $\daleth(X)$ can be easily represented as $$ \daleth(X) = \{ (gx,x) \mid g \in G, x \in X \}. $$ The inclusion $\{ (gx,x) \mid g \in G, x \in X \} \subset \daleth(X)$ is obvious, while for any pair $(g_1x, g_2x) \in \daleth(X)$ it is possible to define $$ (\overline{g}y, y) \in \{ (gx,x) \mid g \in G, x \in X \}, \qquad \mbox{ for } \overline{g} = g_1 g_2^{-1} \mbox{ and }  y = g_2x. $$ This, in turn, informs us that  we can decompose $\daleth(X)$ as the union of ``slices" of the saturated diagonal, i.e.
		$$ \daleth(X) = \displaystyle{\bigcup_{g \in G}} \daleth_g(X), $$ where, for each $g \in G$, we set $$\daleth_g(X) := \{ (gx,x) \mid x \in X \}. $$ This decomposition will be quite useful for the rest of our arguments throughout this section. However, before proceeding further, let us describe the homotopy equivalence between $\daleth(X)$ and the broken path space $\PP_2(X)$.
		
		\begin{lemma}\label{HomotP2Daleth}
			Let $X$ be a $G$ space. There is a homotopy equivalence between $\daleth(X)$ and $\PP_2(X)$.
		\end{lemma}
		\begin{proof}
			Start by noticing that there is an obvious inclusion $$ \iota \colon \daleth(X) \hookrightarrow \PP_2(X) \qquad \iota(gx,x) = (c_{gx},c_{x}). $$ Now, consider a map $f \colon \PP_2(X) \rightarrow \daleth(X)$ defined as $$ f(\gamma_1, \gamma_2) = (\gamma_1(1), \gamma_2(0)).$$ It is immediate to see that the composition $f \circ \iota$ corresponds with $\id_{\daleth(X)}$. For the other composition, we obtain $$ (\gamma_1, \gamma_2) = \iota(\gamma_1(1), \gamma_2(0)) = (c_{\gamma_1(1)}, c_{\gamma_2(0)}). $$ 
			
			Define the map $H \colon \PP_2(X) \times I \rightarrow \PP_2(X)$ by 
			$$ H((\gamma_1, \gamma_2), 0) = (\gamma_1, \gamma_2), \qquad H((\gamma_1, \gamma_2), 1) = (c_{\gamma_1(1)}, c_{\gamma_2(0)}), $$ $$ H((\gamma_1, \gamma_2),t) = (\gamma_1^{t}, \gamma_2^t) \qquad \forall 0 < t < 1 $$ where for all $0 \leq s \leq 1$ we put $$ \gamma_1^t(s)= \gamma_1(t(1-s)+s), \qquad \gamma_2^t(s) = \gamma_2(s(1-t)). $$ One observes that said map defines a homotopy between $\iota \circ f$ and $\id_{\PP_2(X)}$ and, as such, we have the homotopy equivalence $\daleth(X) \simeq \PP_2(X)$.
		\end{proof}
		
	Notice that the above lemma generalizes the homotopy equivalence between $\PP_2(X)$ and $\daleth(X)$ noted by Cadavid-Aguilar and González in \cite{CadavidGonzalez21} for finite free actions into arbitrary group actions.
		
	Throughout the rest of this section, assume that $G$ is a finite group, and $X$ is a compact $G$-ANR. By \cite[Theorem 3.15]{LubawskiMarzan14} this implies, in turn, that the saturated diagonal $\daleth(X)$ becomes a $(G \times G)$-ANR, and we can apply the cohomological Mayer-Vietoris sequence for general subsets that are retractions of open subsets (check, for example, \cite[Pag. 150]{Hatcher}). Also recall that by the \emph{cohomological dimension} of a space $X$ we mean the largest integer $n \geq 0$ such that there exists a local coefficient system $M$ satisfying $H^n(X;M) \neq 0$.
		
		\begin{lemma}\label{LemDimInv}
			Let $X$ be a $G$-CW complex such that $\cd(X^H) \leq \cd(X)$ for all non-trivial subgroups $H \leqslant G$. Then, given any list $L$ of non-trivial subgroups of $G$, we have $$ \cd\displaystyle{ \left(\bigcup_{H \in L} X^H \right)} < \cd(X) + |L| - 1. $$
		\end{lemma}
		
		\begin{proof}
			We will proceed by induction. Consider the base case $|L| = 1$, then $L$ consists of only one non-trivial subgroup $H$ of $G$ and hence $\cd(X^H) \leq \cd(X)$ by the initial hypothesis. 
			
			Now assume that the claim is satisfied for any list of subgroups of cardinality $n-1$, and define $L:=\{K_1, \cdots, K_{n-1}\} \cup \{H\}$ with $H, K_i \leqslant G$ for all $1 \leq i \leq n-1$. Define the sets $$ A := \displaystyle{\bigcup_{K_i \in L}}X^{K_i} \qquad \wedge \qquad B := X^H. $$ Notice that the intersection corresponds to the following union of fixed point sets $$ A \cap B = \displaystyle{\left(\bigcup_{K_i \in L}X^{K_i} \right)} \cap X^H = \displaystyle{\bigcup_{K_i \in L}(X^{K_i} \cap X^H)} = \displaystyle{\bigcup_{K_i \in L} } X.^{\langle K_i, H \rangle} $$ By the induction hypothesis, we have the inequalities $$\cd(A) < \cd(X) + n-2 \qquad \cd(A\cap B) < \cd(X) + n-2,$$ while we also have the inequality $\cd(B) < \cd(X)$ as a consequence of the initial hypothesis. Applying the Mayer-Vietoris sequence to the spaces just defined, and putting $d := \cd(X)+n-2 $, we obtain a sequence
			
			\vspace{0.3cm}	
			
			\adjustbox{scale=0.9, center}{%
				\begin{tikzcd}
					\cdots \rar   & H^{d}(A\cup B;M) \rar & H^{d}(A;M) \oplus H^{d}(B;M) \rar & H^{d}(A\cap B;M)
					\ar[out=0, in=180, looseness=2, overlay]{dll}   & \\
					& H^{d+1}(A\cup B;M) \rar & 0 \rar & 0
			\end{tikzcd}} 
			where $M$ is an arbitrary (possibly twisted) coefficient system and, by the cohomological dimensional bounds stated above, we have that $H^{d+1}(A\cup B;M) = 0$, and thus we obtain that $$ \cd \displaystyle{\left(\bigcup_{K_i \in L} X^{K_i} \right)} \cup X^H) < \cd(X) + |L| - 1. $$ \end{proof}
		
		The lemma above is instrumental, both in the result itself and in the argument of the proof, of the following bound of the cohomological dimension of the saturated diagonal. 
		
		In the same spirit as before, for any list of elements $L \subseteq G$, define the relative saturated diagonal with respect to $L$ as $$ \daleth_L(X)  = \displaystyle{\bigcup_{g_i \in L}} \daleth_{g_i}(X).$$   
		
		\begin{theorem} \label{PropDimDiag1}
			Let $X$ be a $G$-CW complex such that $\cd(X^H) \leq \cd(X)$ for all non-trivial subgroup $H \leqslant G$. Then, for any $L$ list of elements of $G$, $$\cd(\daleth_L(X)) \leq \cd(X) + |L| - 1.$$ In particular we have that $$ \cd(\daleth(X)) \leq \cd(X) + |G| - 1. $$
		\end{theorem}
		\begin{proof}
			The idea of this proof builds upon the argument used in the previous lemma, and we will proceed, once again, by induction. First assume we consider lists consisting on only one element. Then $\daleth_L(X)$ is just homeomorphic to $X$ and, as such, $\cd(\daleth_L(X)) = \cd(X)$.   
			
			Now, let us assume that the induction hypothesis is satisfied for any list of elements of $G$ of length $n-1$. Define an arbitrary list of such length $L' = \{g_1, \cdots, g_{n-1} \}$, and let $L =  L' \cup \{r\}$ for $G' = \{g_1, \cdots, g_{n-1} \}$ a subgroup of order $n-1$ and $r$ an element of $G$ not included in $L'$. Consider the decomposition $$ \daleth_L(X) = \displaystyle{\bigcup_{k_i \in L}} \daleth_{k_i}(X) = \left( \displaystyle{\bigcup_{g_i \in L'}} \daleth_{g_i}(X) \right)  \cup \daleth_{r}(X).  $$	Define now the sets $$ A := \displaystyle{\bigcup_{g_i \in L'}} \daleth_{g_i}(X)  \qquad B := \daleth_r(X). $$ The intersection of these two subsets corresponds with the following set $$ A \cap B = \left(\displaystyle{\bigcup_{g_i \in L'}} \daleth_{g_i}(X) \right) \cap \daleth_r(X) = \displaystyle{\bigcup_{g_i \in L'}} \left( \daleth_{g_i}(X) \cap \daleth_r(X) \right). $$ For each $g_i \in L'$, the intersection $ \daleth_{g_i}(X) \cap \daleth_r(X)$ is equivalent to the set $$ \{x \in X \mid (g_i x,x) = (rx,x)\},$$ which implies $r^{-1}g_i x = x$. Thus we can identify the intersection $$\daleth_{g_i}(X) \cap \daleth_r(X) \cong X^{\langle r^{-1} g_i \rangle }$$ and, consequentially, the intersection above can be reformulated as an union of invariant sets of the form $$ A \cap B = \displaystyle{\bigcup_{g_i \in L'}} X^{\langle r^{-1}g_i \rangle} . $$ Define $M$ as the collection of non-trivial subgroups of $G$ $$M := \{ \langle r^{-1}g_1 \rangle, \cdots, \langle r^{-1}g_{n-1} \rangle  \}. $$ By Lemma \ref{LemDimInv} we know that $$\cd(A\cap B) < \cd(X) + n-2.$$ By the induction hypothesis one observes $$ \cd(A) < \cd(X) + n-2$$ and clearly $\cd(B) = \cd(X)$. Applying now the Mayer-Vietoris sequence as in Lemma \ref{LemDimInv} yields the exact sequence 
			\vspace{0.3cm}	
			
			\adjustbox{scale=0.9, center}{%
				\begin{tikzcd}
					\cdots \rar   & H^{d}(A\cup B;M) \rar & H^{d}(A;M) \oplus H^{d}(B;M) \rar & H^{d}(A\cap B;M)
					\ar[out=0, in=180, looseness=2, overlay]{dll}   & \\
					& H^{d+1}(A\cup B;M) \rar & 0 \rar & 0
			\end{tikzcd} } \vspace{0.3cm}
			and given the cohomological dimensional bounds stated above, we obtain $$\cd(A\cup B) = \cd(\daleth_L(X)) \leq \cd(X) + |L| - 1. $$ which gives us the desired result. \end{proof}
		
		As a consequence of the previous result, we can deduce a cohomological condition on the base space $X$ for non-vanishing second stage effective topological complexity, reflected in the following corollary.		
		
		\begin{corollary}\label{TC2great0}
			Under the assumptions of Theorem \ref{PropDimDiag1}, we have that, if $|G| \leq \cd(X)$, then $$\tc^{G,2}(X) > 0.$$ 
		\end{corollary}
		\begin{proof}
		 Consider the second effective fibration $\pi_2 \colon \PP_2 \rightarrow X \times X$. This map induces an homomorphism in cohomology $$ H^*(X \times X;M) \xrightarrow{\pi_2^*} H^*(\PP_2(X);M).$$ By the homotopy equivalence given in Lemma \ref{HomotP2Daleth}, this homomorphism can be seen as $$ H^*(X \times X;M) \rightarrow H^*(\daleth(X);M) $$ and, by Proposition \ref{PropDimDiag1}, $H^k(\daleth(X);M) = 0$ for any $k > \cd(X) + |G| - 1$. However, $H^{2n}(X \times X) \neq 0$, which implies the existence of at least one non-trivial element in $$\ker(H^*(X \times X;M) \rightarrow H^*(\daleth(X));M).$$ Thus, by (2) in Theorem \ref{TheoremSecatProperties}, we have $$ \secat(\pi_2) = \tc^{G,2}(X) > 0. $$  \end{proof}
		
	\begin{remark}
	In \cite[Definition 7.3]{CadavidGonzalez21} the authors introduced the notion of \emph{effective zero-divisors}. Namely, considering the inclusion of the saturated diagonal $\delta_X \colon \daleth \rightarrow X \times X$, we say that an effective zero-divisor is an element in the kernel of the induced map in cohomology $$\delta^*_X \colon H^*(X \times X;R) \rightarrow H^*(\daleth(X);R)$$ where cohomology is considered with arbitrary coefficients. As noted by Grant in \cite{GrantEquiv24}, it is implicit in \cite{CadavidGonzalez21} that, if $X$ is a free $G$-space with $G$ a finite group, we have the following lower bound $$ \tc^{G,2}(X) = \tc^{G, \infty}(X) \geq \mbox{nil} \ker (\delta^*_X \colon H^*(X \times X;R) \rightarrow H^*(\daleth(X);R)). $$  Our Corollary \ref{TC2great0} generalizes such lower bound (at stage two) to non necessarily free actions with prescribed cohomological dimensional bounds. 
	\end{remark} 
	
	\begin{example}
		Let us consider now the case of our space being a $n$-sphere (for $n > 1$) with a $\ZZ_2$-action by involution and codimension one fixed point set. Adopt the notation $\ZZ_2 = \{ e,g \}$, where $e$ acts as the identity element. As above, take the split saturated diagonal $\daleth(S^n)$ as the union of slices $$\daleth_e(S^n) \bigcup \daleth_g(S^n),$$
		 Similarly as before, the intersection $\daleth_e(S^n) \bigcap \daleth_g(S^n)$ corresponds with the set of elements of $S^n$ such that $x = gx$, which is precisely $(S^n)^{\ZZ_2}$, the set of invariants by the action of $\ZZ_2$. 
		
		By Lemma \ref{HomotP2Daleth}, we know $\PP_2(S^n)$ is homotopically equivalent to the saturated diagonal $\daleth(S^n)$. Given that the dimension of the fixed point set $(S^n)^{\ZZ_2}$ is $n-1$ by the choice of the group action, we are under the hypothesis of Theorem \ref{PropDimDiag1} and thus, by Corollary \ref{TC2great0}, we have that $$\tc^{\ZZ_2,2}(S^n) > 0.$$
		
		The idea of how to find the $(\ZZ_2,2)$-motion planners is essentially analogous to \cite[Proposition 5.6]{BlKa2}. Let us recall it briefly. Define a homeomorphism $\tau: S^n \rightarrow S^n$ by $$ \tau(x_0, \cdots, x_n) := (-x_0, \cdots, -x_{\tilde{r}}, x_{\tilde{r} + 1}, \cdots, x_n) $$ The two-fold motion planners are given over the open covering $$ U_1 = \{(x,y) \in S^n \times S^n \mid y \neq -x\}, $$ $$ U_2 = \{(x,y) \in S^n \times S^n \mid y \neq -\tau(x)\}. $$ The motion planner on $U_1$ is just $s_1 = s'$, where $s'(x,y)$ denotes the shortest arc connecting two non-antipodal points $x$ and $y$. Meanwhile, the motion planner over $s_2 \colon U_2 \rightarrow \PP_2(S^n)$ is defined by putting
		$$s_2(x,y) := (c_x, s'(gx, \tau(x)) \ast s'(\tau(x),y)) \mbox{ for } (x,y) \in U_2$$.
		
		But under this choice of action, by Theorem \ref{EffSpheres}, we actually know that $$ \tc^{\ZZ_2,\infty}(S^n) = \tc^{\ZZ_2,3}(S^n) = 0. $$ Indeed, recall that a $(\ZZ_2,3)$-motion planner over $S^n$ can be defined, as shown in \cite[Proposition 5.7]{BlKa2}, by $$ s(x,y) := (c_x, s'(\Gamma(x)x, N)\ast s'(N, \Gamma(y)y, c_y)) $$ for any $x,y \in S^n$, where $N \in S^n$ denotes the north pole, and $\Gamma(x)$ is the trivial element of $\ZZ_2$ if $x \in S^n_{+}$, and its generator otherwise.  
	\end{example} \medskip
	
	Notice that in the previous example we have just realized the following basic sequence for involutions on arbitrary spheres $S^n$, $n > 1$, with codimension one fixed point set.
	
	\begin{proposition}\label{ExSeqSphInvol}
		Let it $\ZZ_2$ act on $S^n$ by involution, with fixed point set of codimension one. Then, the effective topological complexity sequence associated to $S^n$ is	
		$$ \tc^{\ZZ_2,k}(S^n)  = \begin{cases}
			2, & k=1,\\
			1, & k=2,\\
			0, & k\geq 3.
		\end{cases} $$
	\end{proposition}
	\bigskip
	
\bibliography{PropEffectiveTC.bib}{}
\bibliographystyle{plain}

\end{document}